\documentclass[11pt]{article}

\usepackage{amsfonts,amssymb,amsmath,amsthm,epsfig,euscript}
\usepackage{graphicx}
\usepackage{enumerate}
\usepackage{comment}

\newtheorem{theorem}{Theorem}

\newtheorem{definition}[theorem]{Definition}

\long\def\symbolfootnote[#1]#2{\begingroup
	\def\thefootnote{\fnsymbol{footnote}}\footnote[#1]{#2}\endgroup}

\newcommand{\cref}[1]{Corollary \ref{corollary:#1}}

\newcommand{\red}{\mathrm{red}}

\newcommand{\Pmch}[1]{\text{$P$-$\mathrm{mch}$}(#1)}
\newcommand{\Gammamch}[1]{\text{$\Gamma$-$\mathrm{mch}$}(#1)}

\newcommand{\bounce}{\mathrm{bounce}}
\newcommand{\cross}{\mathrm{cross}}

\newcommand{\fig}[3]
{\begin{figure}[ht]
		\centerline{\scalebox{#1}{\epsfig{file=#2.eps}}}
		\vspace{-1mm}
		\caption{#3}
		\label{fig:#2}
	\end{figure}}

	\title{Paired patterns in lattice paths}
	
	\author{
		Ran Pan \\
		\small Department of Mathematics\\[-0.8ex]
		\small University of California, San Diego\\[-0.8ex]
		\small La Jolla, CA 92093-0112. USA\\[-0.8ex]
		\small \texttt{ran.pan.math@gmail.com}
		\and
		Jeffrey B. Remmel \\
		\small Department of Mathematics\\[-0.8ex]
		\small University of California, San Diego\\[-0.8ex]
		\small La Jolla, CA 92093-0112. USA\\[-0.8ex]
		\small \texttt{remmel@math.ucsd.edu}
		\and
	}
	
	\date{}
	
\begin{document}
		\maketitle
		
		\begin{abstract}
			Let $\mathcal{L}_n$ denote the set of all paths from $[0,0]$ to $[n, n]$ which consist of either unit north steps $N$ or unit east steps $E$ or, equivalently, 
			the set of all words $L \in \{E,N\}^*$ with $n$ $E$'s and 
			$n$ $N$'s. Given $L \in \mathcal{L}_n$ and a subset $A$ of $[n] = \{1, \ldots, n\}$,  we let $ps_{L}(A)$ denote 
the word that results from $L$ by removing 
			the $i^{th}$ occurrence of $E$ and the $i^{th}$ occurrence of $N$ in 
			$L$ for all $i \in [n]-A$, reading from left to right. Then we say that a paired pattern 
			$P \in \mathcal{L}_k$ occurs in $L$ if there is some $A \subseteq [n]$ 
			of size $k$ such that $ps_L(A) = P$. 
			In this paper, we study the  generating functions of paired 
			pattern matching in $\mathcal L_n$.
		\end{abstract}

		\section{Introduction}
		Let $\mathcal L_n$ denote the set of all paths from 
		$[0,0]$ to $[n,n]$ which consist 
		of either unit north $[0,1]$ steps or unit east $[1,0]$ steps.
		The six paths in $\mathcal L_2$ are pictured at the top of 
		Figure \ref{fig:L2F2}. Clearly,
		$$
		|\mathcal L_n|=\binom{2n}{n}.
		$$
		
		We code elements in $\mathcal L_n$  as  words over the alphabet $\{N,E\}$ with 
		$n$ $N$'s and $n$ $E$'s. Given $L \in \mathcal{L}_n$ and a subset $A$ of $[n] = \{1, \ldots, n\}$,  we let $ps_{L}(A)$ 
		denote the word that results from $L$ by removing 
		the $i^{th}$ occurrence of $E$ and the $i^{th}$ occurrence of $N$ in 
		$L$ for all $i \in [n]-A$, reading from left to right.
		For example, suppose $L=NEEENN\in\mathcal L_3$, then 
		$ps_L(\{1\})=NE$, $ps_L(\{2\})=EN$, $ps_L(\{3\})=EN$, 
		$ps_L(\{1,2\}) = ps_L(\{1,3\}) = NEEN$, and $ps_L(\{2,3\}) = EENN$.  
		We shall think of a word in $\{N,E\}$ with 
		$n$ $N$'s and $n$ $E$'s as a {\em paired pattern} where 
		the $i^{th}$ occurrence of $E$ is paired with the $i^{th}$ occurrence 
		of $N$, reading from left to right, for $i =1, \ldots, n$. 
		\begin{definition}
			Given a set of paired patterns $\Gamma \subseteq  \mathcal L_k$ and 
			word $L \in \mathcal{L}_n$, we say that  
			\begin{enumerate}
				\item $\Gamma$ {\bf occurs} in $L$ if there is an $A \subseteq [n]$ of 
				size $k$ such that $ps_L(A) \in \Gamma$, 
				\item there is a {\bf $\Gamma$-match in $L$ starting at the 
					$j^{th}$ paired step}\\ if $ps_L(\{j,j+1,j+2\cdots,j+k-1\}) \in \Gamma$ and
				\item $L$ {\bf avoids} $\Gamma$ if there is no  $\Gamma$-matches in $L$.
			\end{enumerate}
		\end{definition}
		
		Alternatively, we can code a path $L$ as a $2 \times n$ array $T(L)$ 
		where the bottom row of $T$ consists of the positions of the east steps, 
		reading from left to right, and the top row of $T$ consists of the positions 
		of the north steps, reading from left to right. We let $T(L)_{k,1}$ denote the element in the $k^{th}$ column of the bottom row of $T(L)$, and let  $T(L)_{k,2}$ denote the element in the $k^{th}$ column of the top row. Given 
		any $2 \times n$ array $S$ filled with pairwise distinct positive 
		integers, let the reduction of $S$, $\red(S)$, denote the $2 \times n$ 
		array which results 
		from $S$ by replacing the $i^{th}$ smallest integer in $S$ by 
		$i$. An example of the reduction operation $\red$ is pictured 
		at the bottom of Figure \ref{fig:Correspond}.

		\fig{1.0}{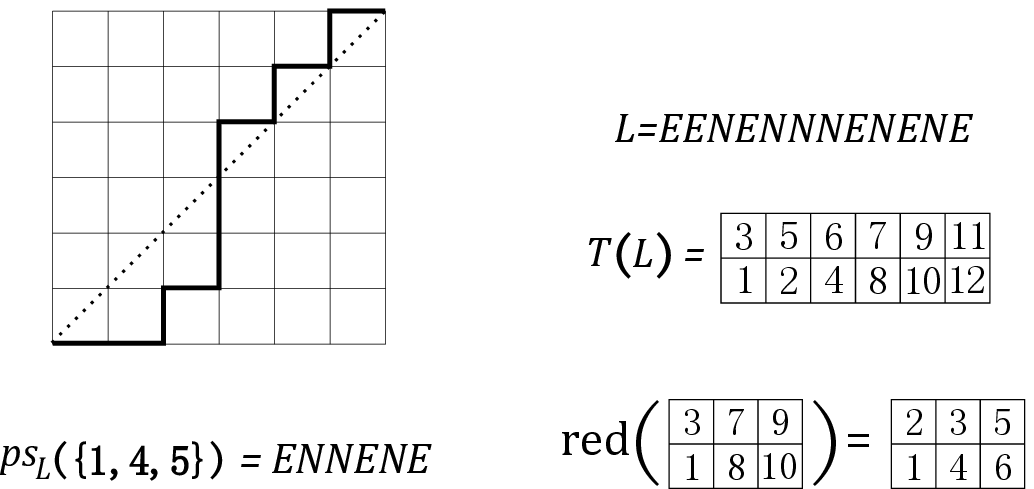}{The correspondence between paths and 
			$2 \times n$ arrays.}

		It is then easy to see that given $L \in \mathcal{L}$ and 
		$A \subseteq [n]$, the array associated with $ps_L(A)$ 
		corresponds to the array obtained by taking the columns 
		in $T(L)$ corresponding to $A$ and reducing.  This process is 
		pictured in Figure \ref{fig:Correspond}. This given, 
		we can restate our pattern matching conditions in terms 
		of $2 \times n$ arrays.  That is, the 
		$\mathcal{T}_n$ denote the set of all $2 \times n$ arrays $T$ filled 
		with the numbers $1,2 \ldots, 2n$ such that the rows of $T$ are 
		increasing reading from left to right. Given $T \in \mathcal{T}_n$ and 
		$A \subseteq [n]$, we let $T[A]$ be the array that results 
		by removing the columns corresponding to elements in $[n]-A$. 
		For example, if $T =T(L)$ is the array pictured in Figure 
		\ref{fig:Correspond}, then $T[\{1,4,5\}]$ is pictured at the bottom 
		left of Figure \ref{fig:Correspond}.

		Then from the point of view of arrays in $\mathcal{T}_n$, our 
		paired pattern matching conditions can be stated as follows. 
		\begin{definition}
			Given a set of $2 \times k$ arrays $\Gamma \subseteq \mathcal{T}_k$ and 
			a $2 \times n$ array $S \in \mathcal{T}_n$, we say that  
			\begin{enumerate}
				\item $\Gamma$ {\bf occurs} in $S$ if there is an $A \subseteq [n]$ of 
				size $k$ such that $\red(S[A]) \in \Gamma$,  
				\item there is a {\bf $\Gamma$-match in $S$ starting at column 
					$j$} if $\red(S[\{j,j+1,j+2\cdots,j+k-1\}]) \in \Gamma$ and
				
				\item $S$ {\bf avoids} $\Gamma$ if there is no  $\Gamma$-matches in $S$.
			\end{enumerate}
		\end{definition}
		
		Note that from this point of view, $\Gamma$-matches correspond naturally 
		to consecutive patterns matches in $2 \times n$ arrays. Results about consecutive patterns in arrays can be found in \cite{HR}.
		We let  $\Gammamch{L}$ denote the number of  $\Gamma$-matches in $L$. 
		If $\Gammamch{L}=0$, then we will say that $L$ has no $\Gamma$-matches. 
		If $\Gamma =\{P\}$ is a singleton, then we will write 
		$\Pmch{L}$ for $\Gammamch{L}$.

		For example, there are six possible patterns of length four, as pictured in Figure \ref{fig:L2F2}, namely, 
		$P_1=EENN,P_2=ENEN,P_3=NEEN,P_4=ENNE,P_5=NENE,P_6=NNEE$.
		
		\fig{1.2}{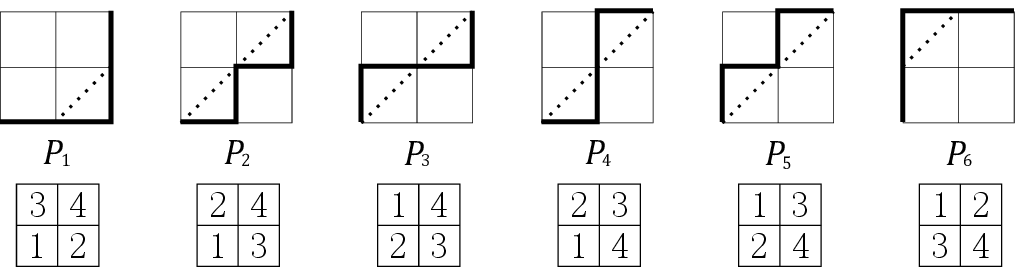}{$\mathcal L_2=\{P_1,P_2,P_3,P_4,P_5,P_6.\}$}
		
		We note that paired patterns differ from classic consecutive patterns 
		in words (e.g. \cite{Deu} \cite{STTM} \cite{STT}). Paired patterns actually describe relationships between paths and the diagonal $y=x$, the 
subdiagonal $y=x-1$, and the superdiagonal  $y=x+1$. For our purposes, 
		the set of Dyck paths $\mathcal{D}_n$ is the set of paths 
		of $\mathcal{L}_n$ which stay weakly below the diagonal $y =x$. 
		For example, in $\mathcal L_2$, the only two Dyck paths are $P_1$ and $P_2$. Actually, a path $L$ is a Dyck path if and only if $L$ has no $(\mathcal L_2-\{P_1,P_2\})$-matches. More details and geometric interpretation of paired patterns can be found in Section 2.
		
		By Theorem \ref{thm:bounce_thm} and Theorem \ref{thm:cross_thm}, we see that some certain paired patterns are equivalent to returns (bouncings) and crossings of a path. These classical statistics  have been studied in literature such as \cite{Moh}, \cite{GoSun}, \cite{SenJain} and \cite{Kra}. 
		
		In this paper, we will focus on paired patterns of length 4 and pattern matching for subsets of these pattern. In other wordsd, we would study generating functions of the form  
		\begin{eqnarray}
		F_{P_k}(x,t)&:=&1+\sum_{n\geq 1}t^n\sum_{L\in\mathcal L_n}x^{P_k\text{-mch}(L)},
		\end{eqnarray}
		where $k\in\{1,2,3,4,5,6\}$, and
		\begin{eqnarray}
		F_{\Delta}(\mathbf{x},t)&:=&1+\sum_{n\geq 1}t^n\sum_{L\in\mathcal L_n}\left(\prod_{j\in \Delta} x_j^{P_j\text{-mch}(L)}\right),
		\end{eqnarray}
		 where $\Delta$ is a subset of $\{1,2,3,4,5,6\}$.

		Note there are two basic symmetries in our study of paired patterns. 
		First one can reflect a path $L \in \mathcal{L}_n$ about the diagonal 
		$y=x$ which has the effect of interchanging $E$'s with $N$'s in the 
		word of $L$ or interchanging the rows in the diagram of 
		$T(L)$ of $L$. Second, one can rotate the path by 180 degrees which 
		has the effect of interchanging the $E$'s and $N$'s and then 
		reversing the word of $L$.  These symmetries immediately 
		show that 
		\begin{eqnarray*}
			F_{P_1}(x,t) &=& F_{P_6}(x,t), \\
			F_{P_2}(x,t) &=& F_{P_5}(x,t), \ \mbox{and} \\
			F_{P_3}(x,t)&=& F_{P_4}(x,t).
		\end{eqnarray*}
		Thus we need only compute three generating functions of the form 
		$F_{P_k}(x,t)$.  
		
		We can also give geometric interpretations to $P_k$-matches for 
		each $k$. For example, we shall show that the number of $P_1$-matches 
		in a path $L \in \mathcal{L}_n$ is the number of east steps that are below 
		the subdiagonal $y =x-1$. The formulas for the generating functions 
		that we will derive, then lead to many interesting bijective problems. 
		For example, we will show that the total number of east steps 
		that lie below the subdiagonal $y =x-1$ over all paths 
		$L \in \mathcal{L}_n$ equals the sum of the areas under  all Dyck paths 
		in $\mathcal D_n$. 
		
		The outline of this paper is as follows. 
		In Section 2, we shall give the geometric interpretations 
		of the number of $P_k$-matches in paths in $\mathcal{L}_n$. 
		In Section 3, we shall derive closed formulas for the generating 
		functions $F_{P_k}(x,t)$ for $k=1, \ldots, 6$ and explore some of 
		the consequences of such formulas. 
		In Section 4, we derive a number of 
		formulas for $F_\Delta(\mathbf{x},t)$ for certain $\Delta \subseteq 
		\{P_1, \ldots, P_6\}$. Finally, in Section 5, we discuss topics for 
		future research such as finding bijections between paths with certain pattern matching condition and other known objects, extending the definition of paired patterns 
		to  Delannoy paths and finding generating functions 
		$F_{P}(x,t)$ for paths $P$ of length greater than $4$. 
		
		\section{The geometric interpretation of the number 
			of $P_k$-matches.}

		In this section, we shall give our geometric interpretations of 
		$P_k$-matches for $k=1, \ldots, 6$. 
		
		\begin{theorem} Let $L \in \mathcal{L}_n$. Then the number 
			of $P_1$-matches in $L$ is the number of east steps below the 
			subdiagonal $y=x-1$.  Hence, by symmetry, the number 
			of $P_6$-matches in $L$ is the number of north steps above the 
			superdiagonal $y=x+1$.
		\end{theorem}
		\begin{proof} Suppose that the $i^{th}$ east step in $L$ 
			occurs below the subdiagonal $y=x-1$ and that this step corresponds 
			to the $j^{th}$ letter in the word $w_1 \ldots w_{2n}$ of $L$.  Then it must be 
			the case that  the number of $E$'s in $w_1 \ldots w_j$ exceeds 
			the number of $N$'s in $w_1 \ldots w_j$ by at least two. This means 
			that when we restrict the diagram $T(L)$ to the letters $1, \ldots, j$, 
			then there are no elements in the $(i-1)^{th}$ and $i^{th}$ columns of the bottom row. This means that $T(L)_{i-1,1}<T(L)_{i,1} < T(L)_{i-1,2}<T(L)_{i,2}$ 
			so that $\red(T(L)[\{i-1,i\}])$ matches the array for $P_1$. 
			Hence each such east step represents a $P_1$-match in $L$.

			On the other hand suppose  $\red(T(L)[\{i-1,i\}])$ matches the array for $P_1$. 
			If $T(L)_{i,1} =j$, then in the word $w =w_1 \ldots w_{2n}$ of 
			$L$, the $j^{th}$ $E$, reading from right to left, must be proceded 
			by at most $i-2$ north steps which means that the east step corresponding 
			to $w_j$ is below the subdiagonal $y =x-1$. 
		\end{proof}

		Given a path $L \in \mathcal{L}_n$, we let 
		$\bounce^{-}(L)$ denote the number of points $[x,x]$ on $L$ 
		such that the preceding step is a north step $N$ and the 
		following step is an east step $E$. This means 
		that the path bounces off the diagonal to the right. 
		We let $\bounce^{+}(L)$ denote the number of points $[y,y]$ on $L$ which 
		is preceded by an east step $E$ and followed by a north step $N$. This means 
		that the path bounces off the diagonal to the left.
		For example,
		for the path $L$ pictured in Figure \ref{fig:bounce}, 
		the points $[6,6]$, $[7,7]$, and $[8,8]$ 
		are points preceded by a north step and followed by an east step so that 
		$\bounce^{-}(L) =3$ and the point $[4,4]$ is preceded by an 
		east step and followed by a north step so that $\bounce^{+}(L) =1$.

		\fig{1.0}{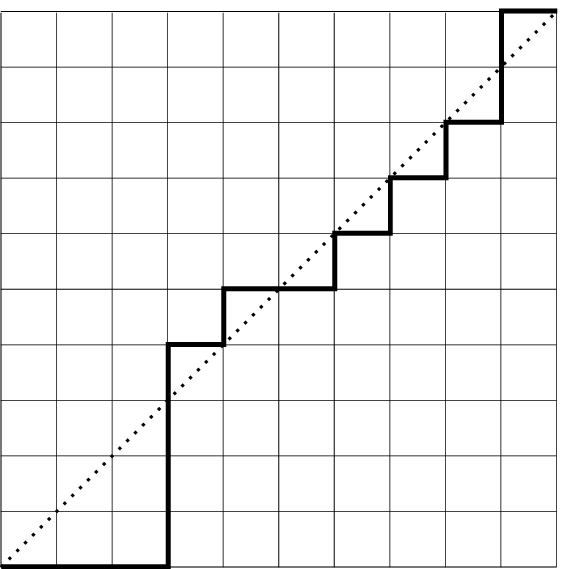}{$\bounce^{+}(L)=1$, $\bounce^{-}(L)=3$, $\cross^h(L)=1$ and $\cross^v(L)=2$.}

		\begin{theorem}\label{thm:bounce_thm} Let $L \in \mathcal{L}_n$. Then the number 
			of $P_2$-matches in $L$ equals $\bounce^{-}(L)$. Hence, 
			by symmetry, the number 
			of $P_5$-matches in $L$ equals $\bounce^{+}(L)$.
		\end{theorem}
		\begin{proof} Consider the diagram $T(L)$ of $L$. 
			Then a $P_2$-match in $L$ corresponds to a pair of columns 
			$i-1$ and $i$ such  $\red(T(L)[\{i-1,i\}])$ matches the array for $P_2$.
			This means that $T(L)_{i-1,1}<T(L)_{i-1,2} < T(L)_{i,1}<T(L)_{i,2}$. Now suppose 
			that $T(L)_{i-1,2} =x$. It follows that all the elements 
			in the columns to the right of $x$ must be greater than $x$ and all 
			the elements in the columns to the left of $x$ must be less than $x$. 
			Since $x > T(L)_{i-1,1}$ it follows that $x =2(i-1)$. Similarly, 
			if $T(L)_{i,1} =y$, then all the elements 
			in the columns to the right of $y$ must be greater than $y$ and all 
			the elements in the columns to the left of $y$ must be less than $y$. 
			Since $y < T(L)_{i-1,1}$ it follows that $y =2(i-1)+1$. This 
			means that in the word of $w =w_1 \ldots w_{2n}$ of $L$, 
			$w_{2(i-1)}= N$ and is preceded by $i$ east steps and 
			$i-1$ north steps so that point $[i,i]$ is on the path of $L$ and 
			is preceded by a north step and followed by an east step. 
			
			Vice versa, if $[i,i]$ is on the path of $L$ and 
			is preceded by a north step and followed by an east step, then 
			it is easy to see that in the array $T(L)$ of $L$, 
			we must have that $T(L)_{i-1,2} =2(i-1)$ and $T(L)_{i,1} =2(i-1)+1$ so 
			that $\red(T(L)[\{i-1,i\}])$ must match $P_2$. 
		\end{proof}

		Given a path $L \in \mathcal{L}_n$, we let 
		$\cross^{h}(L)$ denote the number of points $[x,x]$ on $L$ 
		such that the preceding step is an east step  $E$ and the 
		following step is an east step $E$. This means 
		that the path crosses the diagonal horizontally. 
		We let $\cross^{v}(L)$ denote the number of points $[y,y]$ on $L$ which 
		is preceded by a north step $N$  and followed by a north step $N$. This means 
		that the path crosses the diagonal vertically.
		For example,
		for the path $L$ pictured in Figure \ref{fig:bounce}, there is 
		a horizontal crossing of the diagonal at the point $[5,5]$ so 
		that $\cross^h(L) =1$ and there are vertical crossings at the 
		points $[4,4]$ and $[9,9]$ so that $\cross^v(L) =2$.

		\begin{theorem} \label{thm:cross_thm}Let $L \in \mathcal{L}_n$. Then the number 
			of $P_3$-matches in $L$ equals $\cross^h(L)$. Hence, 
			by symmetry, the number 
			of $P_4$-matches in $L$ equals $\cross^v(L)$
		\end{theorem}
		\begin{proof} Consider the diagram $T(L)$ of $L$. 
			Then a $P_3$-match in $L$ corresponds to a pair of columns 
			$i-1$ and $i$ such that $\red(T(L)[\{i-1,i\}])$ matches the array for $P_3$.
			This means that $T(L)_{i-1,2}<T(L)_{i-1,1} < T(L)_{i,1}<T(L)_{i,2}$. Now suppose 
			that $T(L)_{i-1,1} =x$. It follows all the elements 
			in the columns to right of $x$ must be greater than $x$ and all 
			elements in the columns to the left of $x$ must be less than $x$. 
			Since $x > T(L)_{i-1,2}$ it follows that $x =2(i-1)$. Similarly, 
			if $T(L)_{i,1} =y$. It follows all the elements 
			in the columns to the right of $y$ must be greater than $y$ and all 
			the elements in the columns to the left of $y$ must be less than $y$. 
			Since $y < T(L)_{i-1,1}$ it follows that $y =2(i-1)+1$. This 
			means that in the word of $w =w_1 \ldots w_{2n}$ of $L$, 
			$w_{2(i-1)}= E$ and is preceded by $i-1$ east steps and 
			$i$ north steps so that point $[i,i]$ is on the path of $L$ and 
			is preceded by an east step and followed by an east step. 
			
			Vice versa, if $[i,i]$ is on the path of $L$ and 
			is preceded by an east step and followed by an east step, then 
			it is easy to see that in the array $T(L)$ of $L$, 
			we must have that $T(L)_{i-1,1} =2(i-1)$ and $T(L)_{i,1} =2(i-1)+1$ so 
			that $\red(T(L)[\{i-1,i\}]$ must match $P_3$. 
		\end{proof}

\section{Generating functions}
Let $F_i(x,t) = F_{P_i}(x,t)$ for $i=1, \ldots, 6$.
The goal of this section is to compute 
the generating functions $F_k(x,t)$ for $k =1, \ldots, 6$.

To obtain a recurrence for Dyck paths, the usual way is to factorize Dyck paths based on where it returns to the diagonal for the first time. Application of this decomposition can be found in many papers focused on lattice path enumeration such as \cite{Deu}, \cite{GoSun} and \cite{STT}. We shall show that similar ideas 
allow us to obtain recurrences for the number of $P_k$-matches.

\subsection{Pattern $P_1$}\label{sec:P1}
For pattern $P_1$, consider the ordinary generating function $F_1(x,t)$ as follows,
\begin{equation}
F_{1}(x,t):=1+\sum_{n\geq 1}t^n\sum_{L\in \mathcal L_n}x^{P_1\text{-mch}(L)}.
\end{equation}
We know for a path $L$, $P_1\text{-mch}(L)$ is equal to the number of east steps below subdiagonal $y=x-1$. By our observation in 
the introduction, $F_1(x,t) = F_6(x,t)$.

\fig{1.2}{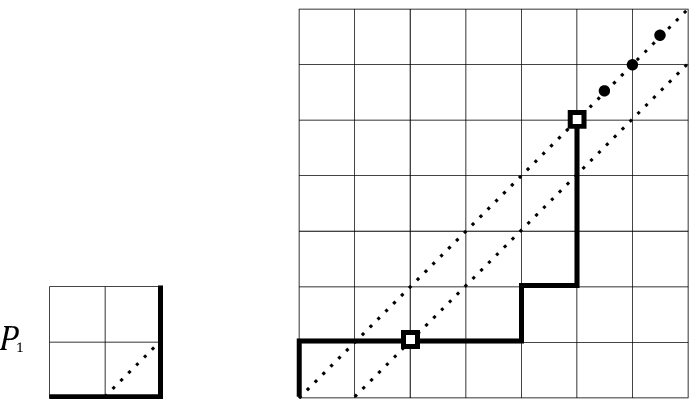}{An example of recurrence based on $P_1$.}

We split the analysis of $P_1\text{-mch}(L)$ into two cases. Case 1 is when $P_1\text{-mch}(L)=0$, that is, path $L$ stays above $y=x-1$. It is easy to see that the number of paths in $\mathcal L_n$ above $y=x-1$ is $C_{n+1}=\frac{1}{n+2}\binom{2n+2}{n+1}$, the $(n+1)^{th}$ Catalan number. Case 2 is when $P_1\text{-mch}(L)>0$, that is, path $L$ has at least one east step below $y=x-1$. 
Now consider the first time the path touches $y=x-1$ and the first time 
after that where the path hits a point $[i,i]$ on the diagonal. 
It is easy to see that the two steps preceding the point $[i,i]$ must 
be north steps. An example of recurrence is pictured in Figure \ref{fig:P1}, where two boxes are the two positions mentioned above and three diagonal dots stand for a whatever path follows the second box.

Suppose the position of the first box has coordinates $[j,j-1]$, $j\geq 1$, clearly there are $C_j$ ways to choose steps before reaching $[j,j-1]$. Similarly, suppose the position of the second box has coordinates $[i+j,i+j]$, $i\geq 1$, clearly there are $C_i$ ways to choose steps between $[j,j-1]$ and $[i+j,i+j]$.

Since the ordinary generating function for Catalan numbers is
\begin{equation}\label{eq:cat}
C(x)=\sum_{n\geq 0}C_nx^n=\frac{1-\sqrt{1-4x}}{2x}=1+x+2x^2+5x^3+14x^4+42x^5\cdots,
\end{equation}
it follows that 
\begin{eqnarray*}
	F_{1}(x,t)&=&\sum_{n\geq 0}C_{n+1}t^n+\sum_{i\geq 1}\sum_{j\geq 1}C_iC_jx^it^{i+j}F_1(x,t)\\
	&=&\frac{C(t)-1}{t}+\sum_{i\geq 1}C_i(xt)^i\sum_{j\geq 1}C_jt^{j}F_1(x,t)\\
	&=&\frac{C(t)-1}{t}+\sum_{i\geq 1}C_i(xt)^i(C(t)-1)F_1(x,t)\\
	&=&\frac{C(t)-1}{t}+(C(xt)-1)(C(t)-1)F_1(x,t)
\end{eqnarray*}
and therefore by Equation (\ref{eq:cat}),
$$
F_1(x,t)=\frac{C(t)-1}{t\left(1-(C(xt)-1)(C(t)-1)\right)}
=\frac{2 x}{x\sqrt{1-4 t} +\sqrt{1-4 xt }+x-1}.
$$
Some initial terms of $F_1(x,t)$ are
\begin{eqnarray*}
	F_1(x,t)&=&1+2 t+(x+5)t^2 + (2x^2+4 x+14)t^3
	+(5 x^3+9 x^2+14 x+42)t^4 \\
	&& +(14 x^4+24 x^3+34 x^2+48 x+132) t^5 \\
	&&+(42 x^5+70 x^4+95 x^3+123 x^2+165 x+429)t^6 +\cdots.
\end{eqnarray*}
The number of paths in $\mathcal L_n$ avoiding $P_1$ is $C_{n+1}$, $(n+1)^{th}$ Catalan number. In general, the number of paths having exactly $k$ $P_1$-matches has the generating function as follows,
$$
\frac{1}{k!}\left.\frac{\partial^k F_1(x,t)}{\partial x^k}\right|_{x=0}~~\text{ or }~~~~
\frac{1}{k!}\left.\frac{\partial^k F_1(x,t)}{\partial x^k}\right|_{x\rightarrow 0}.
$$
We evaluate the derivative at $x=0$ or when $x=0$ is a singularity of the derivative, we take the limit as $x$ approaches zero. For example, 
\begin{eqnarray*}
	\left.\frac{\partial F_1(x,t)}{\partial x}\right|_{x\rightarrow 0}&=&\left(\frac{-1+\sqrt{1-4t}+2t}{2t}\right)^2\\
	&=&t^2+4t^3+14t^4+48t^5+165t^6+572t^7+2002t^8+\cdots.
\end{eqnarray*}
The sequence $1,4,14,48,165,572,202, \ldots$ is sequence 
A002057 in the OEIS \cite{Slo}. It has a number of combinatorial 
interpretations including the number of standard tableaux of 
shape $(n+2,n-1)$ and, with an offset of 4, the number of 123-avoiding permutations on $\{1,2,\cdots,n\}$ for which the integer $n$ is in the fourth spot. 
It follows from the hook length formula for the number of standard 
tableaux that the number of paths $L$ in $\mathcal{L}_n$ 
with exactly one east below the subdiagonal $y=x-1$ equals 
$4((2n-1)!)/((n-2)!(n+2)!)$ and is equal to 
the number of 123-avoiding permutations on $\{1,2,\cdots,n+2\}$ for which the integer $n$ is in the fourth spot. 

Similarly, one can obtain the generating function for the number of paths having exactly two east steps below the subdiagonal as follows,
\begin{eqnarray*}
	\frac{1}{2!}\left.\frac{\partial^2 F_1(x,t)}{\partial x^2}\right|_{x\rightarrow 0}&=&-\frac{\left(-1+\sqrt{1-4t}-2t\right)\left(-1+\sqrt{1-4t}+2t\right)^2}{8t^2}\\
	&=&2t^3+9t^4+34t^5+123t^6+440t^7+1573t^8 + 5642t^9+ \cdots.
\end{eqnarray*}

The sequence $2,9,34,123,440,1573,5642,\cdots$ is sequence 
A120989 in the OEIS \cite{Slo}.  The $n^{th}$ term 
in this sequence counts the level of the first leaf in preorder of a binary tree, summed over all binary  trees with $n-2$ edges. Thus the number of 
paths $L$ in $\mathcal L_n$ with exactly two east steps below the subdiagonal 
$y=x-1$ equals the sum of the level of the first leaf in preorder over 
all binary trees with $n-2$ edges. We leave open the problem of  
giving a bijective proof of this fact. 

Next, we shall answer the following question, for a random path $L\in\mathcal L_n$, what is the expectation of $P_1\text{-mch}(L)$, or in other words, on average how many east steps of $L$ are below $y=x-1$?
Consider that
\begin{eqnarray}
\left.\frac{\partial F_1(x,t)}{\partial x}\right|_{x=1}&=&
-\frac{-1+\sqrt{1-4t}+2t}{2(1-4t)^{3/2}}\\
&=&t^2+8t^3+47t^4+244t^5+1186t^6+5536t^7+\cdots. \label{coeff:P1}
\end{eqnarray}
For example, a random $L\in\mathcal L_7$, expectation of $P_1\text{-mch}(L)$
$$
\mathbb{E}[P_1\text{-mch}(L):L \in \mathcal{L}_7]=\frac{5536}{\binom{14}{7}}\approx 1.63,
$$
which implies in average there are roughly 1.63 east steps below $y=x-1$.

In general, by the OEIS, the coefficient of $t^n$ in Equation (\ref{coeff:P1}) has formula $\frac{1}{2}((n+1)\binom{2n}{n}-4^n)$. Using Stirling's formula to approximate $n!$, we have
\begin{equation}\label{eq:AsyP1}
\mathbb{E}[P_1\text{-mch}(L):L \in \mathcal{L}_n]=
\frac{(n+1)\binom{2n}{n}-4^n}{2\binom{2n}{n}}\sim \frac{n+1}{2}-\sqrt{\pi n},
\end{equation}
which implies when $n$ is large, for a random path $L\in\mathcal L_n$, the expected number of east steps that lie below $y=x-1$ is $\frac{n+1}{2}-\sqrt{\pi n}$.

The sequence $1,8,47,244,1186,5536, \cdots $ from Equation (\ref{coeff:P1}) is sequence 
A029760 and A139262 in the OEIS \cite{Slo}. A029760 and A139262 count the total area under all the Dyck paths from $[0,0]$ to $[n,n]$, the total number of inversions in all 132-avoiding permutations of length $n$ and also total number of two-element anti-chains over all ordered trees on $n$ edges. Again we leave open the problem of finding a bijective proof of these facts. 
We suspect that finding a bijective proof is a challenge because Dyck paths, 132-avoiding permutations and ordered trees are all Catalan objects while lattice paths 
in $\mathcal{L}_n$ are not.

Next, by manipulating $F_1(x,t)$ we can also find the number of paths having even number many east steps below the subdiagonal $y=x-1$. The generating function is as follows,
\begin{eqnarray*}
	&&\frac{1}{2}\left(F_1(1,t)+F_1(-1,t)\right)\\
	&=&1+2t+5t^2+16t^3+51t^4+180t^5+622t^6+2288t ^7\cdots.
\end{eqnarray*}
Similarly, the generating function for the number of paths having odd number many east steps below the subdiagonal $y=x-1$ is 
\begin{eqnarray*}
	&&\frac{1}{2}\left(F_1(1,t)-F_1(-1,t)\right)\\
	&=&t^2+4t^3+19t^4+72t^5+302t^6+1144t^7+4643t^8\cdots.
\end{eqnarray*}

Neither of the series correspond to entries in the OEIS \cite{Slo}.

\subsection{Pattern $P_2$}\label{sec:P2}
For pattern $P_2$, $P_2\text{-mch}(L)$ counts the number of times $L$ bounces off the diagonal $y=x$ to the right, in other words, $P_2\text{-mch}(L)=\bounce^-(L)$. We shall study
\begin{equation}
F_2(x,t):=1+\sum_{n\geq 1}t^n\sum_{L\in \mathcal L_n}x^{P_2\text{-mch}(L)}.
\end{equation}
As we observed in the introduction, $F_2(x,t) =F_5(x,t)$. 

We shall consider two cases. Case 1 are the paths that start with an east step and Case 2 are the paths that start with a north step. 
We define 
\begin{equation*}
G_2(x,t):=\sum_{n\geq 1}t^n\sum_{L\in \mathcal L_n\text{ starting with }E}x^{P_2\text{-mch}(L)}
\end{equation*}
and
\begin{equation*}
H_2(x,t):=\sum_{n\geq 1}t^n\sum_{L\in \mathcal L_n\text{ starting with }N}x^{P_2\text{-mch}(L)}.
\end{equation*}
\fig{1.2}{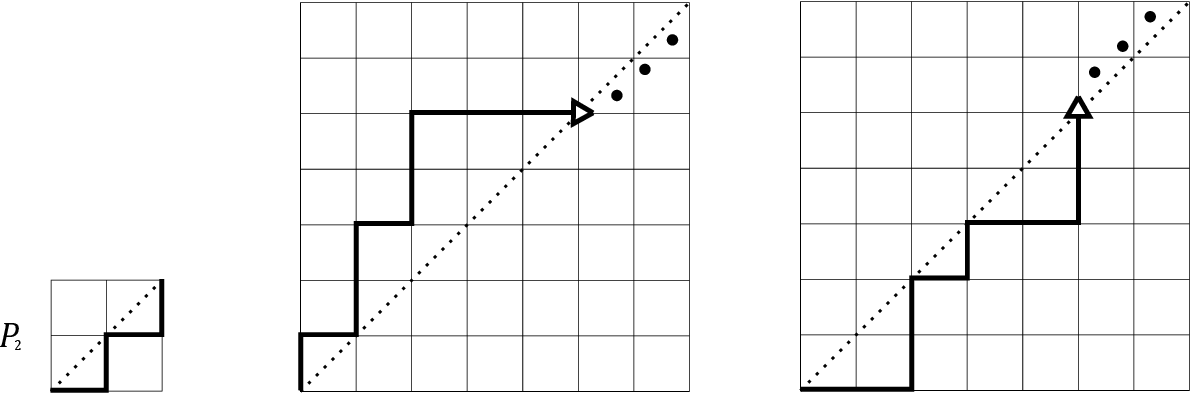}{An example of recurrence based on $P_2$.}
Clearly, $F_2(x,t)=1+G_2(x,t)+H_2(x,t)$. For $H_2(x,t)$, we consider where is the first time a path starting with a north step crosses the diagonal $y=x$ horizontally. In the middle diagram of Figure \ref{fig:P2},  the three dots stand for a path starting with `$E$' or an empty path.
\begin{equation}\label{eq:H_2}
H_2(x,t)=\sum_{j\geq 1}C_jt^j\left(G_2(x,t)+1\right)=(C(t)-1)(G_2(x,t)+1).
\end{equation}
Similarly, for $G_2(x,t)$, we consider where is the first time a path starting with an east step crosses the diagonal $y=x$ vertically. In the right diagram of Figure \ref{fig:P2}, three dots stand for a path starting with `$N$' or an empty path. Since we want to keep track of $P_2$-matches, here we need to introduce Catalan's triangle $C_{i,j}$, which is the number of Dyck paths in $\mathcal L_{2j}$ with $i$ returns to the diagonal \cite{Deu}. By \cite{Deu}, $C_{i,j}$ has generating function as follows, 
\begin{equation}\label{eq:C(x,t)}
C(x,t)=\sum_{i\geq 0}\sum_{j\geq 0}C_{i,j}x^it^j=1+\frac{1-\sqrt{1-4 t}}{(\sqrt{1-4 t}-1) x+2}
\end{equation}

Then
\begin{equation}
G_2(x,t)=\sum_{i\geq 0}\sum_{j\geq 1}C_{i,j}x^it^j\left(H_2(x,t)+1\right)=\frac{1-\sqrt{1-4 t}}{(\sqrt{1-4 t}-1) x+2}(H_2(x,t)+1).
\end{equation}
By Equation (\ref{eq:H_2}),
\begin{equation}
G_2(x,t)=\frac{1-\sqrt{1-4 t}}{(\sqrt{1-4 t}-1) x+2}((C(t)-1)(H_2(x,t)+1)+1).
\end{equation}
We can then solve $G_2(x,t)$ to obtain that 
\begin{eqnarray*}
	G_2(x,t)&=&\frac{\frac{1-\sqrt{1-4 t}}{(\sqrt{1-4 t}-1) x+2}C(t)}{1+\frac{1-\sqrt{1-4 t}}{(\sqrt{1-4 t}-1) x+2}-\frac{1-\sqrt{1-4 t}}{(\sqrt{1-4 t}-1) x+2}C(t)}\\
	&=&\frac{(\sqrt{1-4 t}-1)^2}{2 (\sqrt{1-4 t} (t (x-1)+1)-t (x-5)-1)}.
\end{eqnarray*}
Then we have
\begin{eqnarray*}
	F_2(x,t)&=&1+G_2(x,t)+H_2(x,t)\\
	&=&1+G_2(x,t)+(C(t)-1)(G_2(x,t)+1)\\
	&=&(G_2(x,t)+1)C(t)\\
	&=&\frac{-(\sqrt{1-4 t}-1) (\sqrt{1-4 t} x-\sqrt{1-4 t}-x+3)}{2 (\sqrt{1-4 t} xt-xt-\sqrt{1-4 t} t+5 t+\sqrt{1-4 t}-1)}.
\end{eqnarray*}

A few initial terms of $F_2(x,t)$ are
\begin{eqnarray*}
	F_2(x,t)&=&1+2t+(x+5)t^2+(x^2+4x+15)t^3+(x^3+5x^2+16x+48)t^4\\
	&&+(x^4+6x^3+23x^2+62x+160)t^5+\cdots.
\end{eqnarray*}
$F_2(0,t)$ is the generating function for the number of paths that do not bounce off the diagonal to the right. One can compute that 
\begin{eqnarray*}
	F_2(0,t)&=&\frac{2 (t+\sqrt{1-4 t}-1)}{(\sqrt{1-4 t}-5) t-\sqrt{1-4 t}+1}\\
	&=&1+2t+5t^2+15t^3+48t^4+160t^5+548t^6+1914t^7+ \cdots.
\end{eqnarray*}
The sequence $1,2,5,15,48, 160, 548, 1914 \cdots $ does not appear 
in the OEIS \cite{Slo}. 


Similarly, we can compute the generating 
function of  the number of paths that bounce at diagonal to right 
exactly one time. That is, 
\begin{eqnarray*}
	\left.\frac{\partial F_2(x,t)}{\partial x}\right|_{x= 0}&=&\left(\frac{-1+\sqrt{1-4t}+2t}{1-\sqrt{1-4t}+\left(-5+\sqrt{1-4t}\right)t}\right)^2\\
	&=&t^2+4t^3+16t^4+62t^5+238t^6+910t^7+\cdots.
\end{eqnarray*}
The sequence $1,4,16,62,238,910  \cdots $ does not appear 
in the OEIS \cite{Slo}.

Also we could ask, for a random path $L\in\mathcal L_n$, what is the expectation of $P_2\text{-mch}(L)$, or in other words, on average how many times do $L$ bounce at $y=x$ to right?
Consider that
\begin{eqnarray}
\left.\frac{\partial F_2(x,t)}{\partial x}\right|_{x=1}&=&
\left(\frac{-1+\sqrt{1-4t}+2t}{-1+\sqrt{1-4t}+4t}\right)^2\\
&=&t^2+6t^3+29t^4+130t^5+562t^6+2380t^7+\cdots. \label{coeff:P2}
\end{eqnarray}
Coefficient of $t^n$ in Equation (\ref{coeff:P2}) agrees with 
sequence A008549 of the OEIS \cite{Slo} which counts the total area of all the Dyck excursions of length $2n-2$. By OEIS \cite{Slo}, the coefficient of $t^n$ is given by the 
 formula $4^{n-1}-\binom{2n-1}{n-1}$. Using Stirling's formula to approximate $n!$, one 
finds that 
$$
\mathbb{E}[P_2\text{-mch}(L):L\in\mathcal L_n]=
\frac{\sum_{L \in \mathcal{L}_n} \bounce^-(L)}{|\mathcal L_n|}=\frac{4^{n-1}-\binom{2n-1}{n-1}}{\binom{2n}{n}} \sim
\frac{\sqrt{\pi n}}4 - \frac{1}{2} \approx 0.443 \sqrt{n},
$$
which implies when $n$ is large, the expected number of times a random path $L\in\mathcal L_n$ bounces off the diagonal to the right is roughly $0.443\sqrt{n}$.

Next, by manipulating $F_2(x,t)$ we can also find the number of paths having even number of bounces off the diagonal to the right. The generating function is as follows,
\begin{eqnarray*}
	&&\frac{1}{2}\left(F_2(1,t)+F_2(-1,t)\right)\\
	&=&1+2t+5t^2+16t^3+53t^4+184t^5+654t^6+2368t ^7\cdots.
\end{eqnarray*}
Similarly, the generating function for the number of paths having odd number 
of bounces off the diagonal to the right is  
\begin{eqnarray*}
	&&\frac{1}{2}\left(F_2(1,t)-F_2(-1,t)\right)\\
	&=&t^2+4t^3+17t^4+68t^5+270t^6+1064t^7+4181t^8\cdots.
\end{eqnarray*}
Again, neither of the series correspond to sequences in the 
OEIS \cite{Slo}.

\subsection{Pattern $P_3$}\label{sec:bije}
For pattern $P_3$, as discussed in Section 2, $P_3\text{-mch}(L)$ counts the number of times $L$ crosses the diagonal $y=x$ horizontally. We shall study
\begin{equation}
F_{3}(x,t):=1+\sum_{n\geq 1}t^n\sum_{L\in \mathcal L_n}x^{P_3\text{-mch}(L)}.
\end{equation}
By our observation in the introduction $F_3(x,t) = F_4(x,t)$.

Similar to the discussion of $P_2$, we consider two cases. Case 1 are the paths that start with a north step and Case 2 are the paths that start with an east step. We define 
\begin{equation*}
G_3(x,t):=\sum_{n\geq 1}t^n\sum_{L\in \mathcal L_n\text{ starting with }E}x^{P_3\text{-mch}(L)}
\end{equation*}
and
\begin{equation*}
H_3(x,t):=\sum_{n\geq 1}t^n\sum_{L\in \mathcal L_n\text{ starting with }N}x^{P_3\text{-mch}(L)}.
\end{equation*}
Clearly, $F_{3}(x,t)=1+G_3(x,t)+H_3(x,t)$. Essentially, the way we shall decompose the paths in this case 
is the same as how we decomposed the paths for pattern $P_2$. For paths starting with a north step, we consider where is the first 
the path crosses the diagonal $y=x$ from left to right and then it is followed by a path starting with an east step or an empty path. Then
\begin{equation}\label{eq:H_3}
H_3(x,t)=\sum_{j\geq 1}C_jt^j\left(xG_3(x,t)+1\right)=(C(t)-1)(xG_3(x,t)+1)
\end{equation}
Similarly, for paths starting with an east step , we consider where is the first time that the path crosses the diagonal  vertically. Then
\begin{equation}
G_3(x,t)=\sum_{j\geq 1}C_{j}t^j\left(H_3(x,t)+1\right)=(C(t)-1)(H_3(x,t)+1).
\end{equation}
Then by Equation (\ref{eq:H_3}),
$$
H_3(x,t)=\left(C(t)-1\right)\left(x(C(t)-1)(H_3(x,t)+1)+1\right)
$$
and we solve the formula above for $H_3(x,t)$
\begin{eqnarray*}
	H_3(x,t)&=& \frac{(1-C(t)) (x(C(t)-1)+1)}{x(C(t)-1)^2-1}\\	
	&=&-
	\frac{\left(2 t+\sqrt{1-4 t}-1\right) \left(2 t (x-1)+(\sqrt{1-4 t}-1) x\right)}{2 \left(2 t^2 (x-1)+2 \left(\sqrt{1-4 t}-2\right) t x-\sqrt{1-4 t} x+x\right)}
\end{eqnarray*}

Therefore,
\begin{eqnarray*}
	F_3(x,t)&=&1+G_{3}(x,t)+H_{3}(x,t)\\
	&=&1+C(t)H_3(x,t)+C(t)-H_3(x,t)-1+h_N(x,t)\\
	&=&(H_3(x,t)+1)C(t)\\
	&=&\frac{2}{(2 t (x-1)+(\sqrt{1-4 t}-1) x+\sqrt{1-4 t}+1)}
\end{eqnarray*}

A few initial terms of $F_3(x,t)$ are
\begin{eqnarray*}
	&&F_3(x,t)\\
	&=&1+2t+(x+5)t^2+(6x+14)t^3+(x^2+27x+42)t^4+(10x^2+110x+132)t^5+\cdots.
\end{eqnarray*}
Next, we shall find the generating function of 
the number of paths crossing the diagonal horizontally 
exactly once.
\begin{eqnarray*}
	\left.\frac{\partial F_3(x,t)}{\partial x}\right|_{x= 0}&=&-\frac{2\left(-1+\sqrt{1-4t}+2t\right)}{\left(1+\sqrt{1-4t}-2t\right)^2}\\
	&=&t^2+6t^3+27t^4+110t^5+429t^6+1638t^7+\cdots,
\end{eqnarray*}
The sequence $1,6,27,110,429,1638,\cdots$ is sequence A003517 on OEIS \cite{Slo}.
This sequence has several combinatorial interpretations such 
as the number of standard tableaux of shape $(n+3,n-2)$ and 
the number of permutations of $\{1, \ldots, n+1\}$ with exactly 
one increasing subsequence of length 3. It follows from 
the hook length formula for the number of standard tableaux that 
the number of paths $L$ in $\mathcal{L}_n$ with exactly one horizontal 
crossing equal $6((2n+1)!)/(n-2)!(n+4)!)$.

Similarly, the number of paths $L$ in $\mathcal{L}_n$ with exactly 
2 horizontal crossings has the following generating function:
\begin{eqnarray*}
	\frac{1}{2!}\left.\frac{\partial^2 F_3(x,t)}{\partial x^2}\right|_{x\rightarrow 0}&=&\frac{4\left(-1+\sqrt{1-4t}+2t\right)}{\left(1+\sqrt{1-4t}-2t\right)^2}\\
	&=&t^4+10t^5+65t^6+350t^7+1700t^8+7752t^9\cdots,
\end{eqnarray*}
The sequence $1,10,65,350,1700, \cdots$ is sequence A003519 on OEIS \cite{Slo}. 
It counts the number of standard tableaux of shape $(n-5,n-4)$ 
from which it follows that the 
number of paths $L$ in $\mathcal{L}_n$ with exactly 
2 horizontal crossings equals $\frac{10}{n+6}\binom{2n+1}{n-4}$.

Also we could ask, for a random path $L\in \mathcal{L}_n$, what is the expectation of $P_3\text{-mch}(L)$, or in other words, on average how many times does $L$ cross $y=x$ from left to right? In this case, we have  computed  
that 
\begin{eqnarray*}
	\frac{\partial}{\partial x} F_3(x,t)|_{x =1} &=&   
	\frac{-1=\sqrt{1-4t}+2t}{-2+8t} \\
	&=& t^2 +6t^3 + 29t^4 +130 t^5 +562 t^6 +2880t^7 + 9949t^8 + \cdots\\
	&=& 
		\frac{\partial}{\partial x} F_2(x,t)|_{x =1},
\end{eqnarray*}
which means  the total number of $P_3$-matches in paths in $\mathcal{L}_n$
is equal to the total number of $P_2$-matches  paths in $\mathcal{L}_n$.

Next we give a bijection that shows this fact. Since the total number of $P_3$-matches 
in paths in $\mathcal{L}_n$  is half of total $\{P_3,P_4\}$-matches in paths in $\mathcal{L}_n$
and the total number of $P_2$-matches in paths in $\mathcal{L}_n$  is half of total $\{P_2,P_5\}$-matches in paths in $\mathcal{L}_n$, we only need to show that the total number of $\{P_2,P_5\}$-matches in paths in $\mathcal{L}_n$ is equal to the total number of 
$\{P_3,P_4\}$-matches in paths in $\mathcal{L}_n$. In other words, we only need to show that the total number of times that all the paths in $\mathcal{L}_n$ bounce off the diagonal is equal to the total number of times that all the paths in $\mathcal{L}_n$ cross the diagonal.

By the reflection principle, we can construct a bijection between the set of paths in $\mathcal L_n$ crossing the diagonal $k$ times, denoted by $\mathcal C_{n,k}$ and the set of paths in $\mathcal L_n$ bouncing off the diagonal $k$ times, denoted by $\mathcal B_{n,k}$.  The procedure of the bijection is as follows. For any path $L\in\mathcal C_{n,k}$, $L$ crosses the diagonal $k$ times and suppose $L$ touches the diagonal $j$ times at positions $\{p_1,p_2,\cdots,p_j\}$, $j\geq k$.  First we retain the part between $[0,0]$ and $p_1$ of the path and
 flip the path between $p_1$ and $[n,n]$ along the diagonal, then we can get a new path $L_1$. At the second step, we retain the part between $[0,0]$ and $p_2$ of the path $L_1$ and flip the part between $p_2$ and $[n,n]$ along the diagonal, then we can get a new path $L_2$. We repeat the process above until we acquire $L_j$. $L_j$ is a path in $\mathcal B_{n,k}$ because  the procedure above transforms a crossing of $L$ into a bouncing of $L_j$ and  a bouncing of $L$ into a crossing of $L_j$. An example is pictured in Figure \ref{fig:Reflect}. $L\in\mathcal C_{5,2}$ is mapped to $L_3\in\mathcal B_{5,2}$ under the bijection.

\fig{0.9}{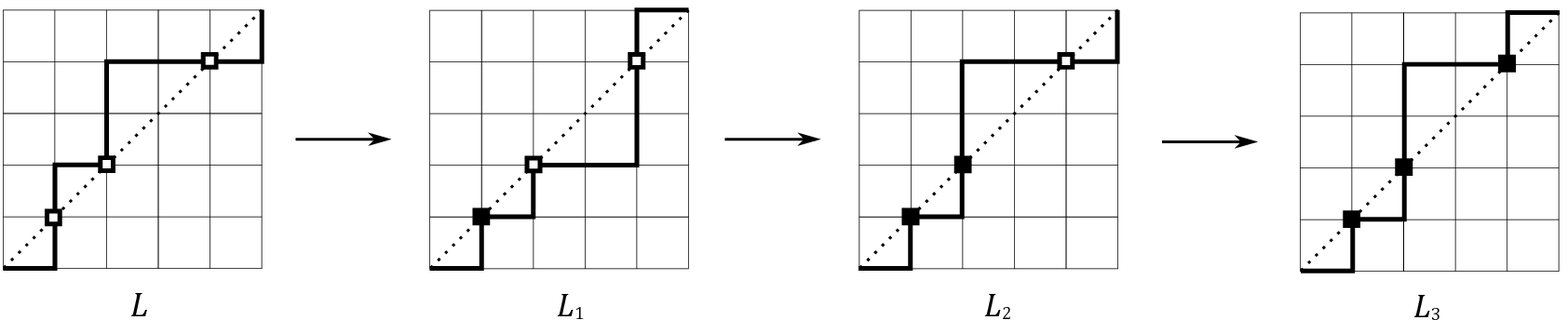}{$L$ is mapped to $L_3$ by the bijection.}
Therefore,
$$
\mathbb{E}[P_3 \text{-mch}(L)]=\mathbb{E}[P_2 \text{-mch}(L)]\sim \frac{\sqrt{n\pi}}{4}-\frac{1}{2}\approx 0.443\sqrt{n}.
$$

Next, by manipulating $F_3(x,t)$ we can also find the number of paths having even number many  horizontal crossings. The generating function is as follows,
\begin{eqnarray*}
	&&\frac{1}{2}\left(F_3(1,t)+F_3(-1,t)\right)\\
	&=&1+2t+5t^2+14t^3+43t^4+142t^5+494t^6+1780t ^7\cdots,
\end{eqnarray*}
The sequence $1,2,5,14,43,142,494, \ldots $ is sequence 
A005317 in the OEIS \cite{Slo} where no combinatorial interpretation is given. Thus we 
have given a combinatorial interpretation to this sequence. 

Similarly, the generating function for the number of paths having odd number many  horizontal crossings is 
\begin{eqnarray*}
	&&\frac{1}{2}\left(F_3(1,t)-F_3(-1,t)\right)\\
	&=&t^2+6t^3+27t^4+110t^5+430t^6+1652t^7+6307t^8\cdots,
\end{eqnarray*}
in which coefficient of $t^n$ also counts number of unordered pairs of distinct length-$n$ binary words having the same number of $1$'s according to A108958 in the OEIS \cite{Slo}. We leave open the problem of giving a bijective proof of this fact.

\section{Multivariate generating functions}
In this section, we shall study multivariate generating functions for $\Delta$-matches for certain 
$\Delta \subseteq \{P_1, \ldots, P_6\}$. Our choices for the $\Delta$ that we consider are motivated by picking 
those pattern matching conditions which have the clearest geometric interpretations.  
Let 
$$
F_{\Delta}(\mathbf{x},t):=1+\sum_{n\geq 1}t^n\sum_{L\in\mathcal L_n}\left(\prod_{j\in \Delta} x_j^{P_j\text{-mch}(L)}\right),
$$
where $\Delta$ is a subset of $\{1,2,3,4,5,6\}$.
We start by looking at the two elements sets that have symmetry, namely, $\Delta = \{1,6\}$, $\Delta = \{2,5\}$, and 
$\Delta = \{3,4\}$.

\subsection{$P_1$ and $P_6$}
Pattern $P_1$ has one east step below $y=x-1$ and $P_6$ has one east step above $y=x+1$. We know that for a path $L\in\mathcal L_n$, $P_1\text{-mch}(L)$ and $P_6\text{-mch}(L)$ are the numbers of east steps below $y=x-1$ and above $y=x+1$, respectively.

\fig{1.3}{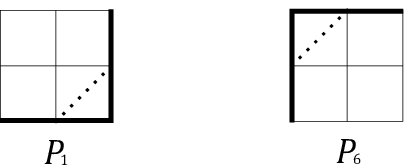}{Pattern $P_1$ and $P_6$.}

In this subsection, we shall consider the multivariate generating function
$$
F_{1,6}(x_1,x_6,t):=1+\sum_{n\geq 1}t^n\sum_{L\in\mathcal L_n}x_1^{P_1\text{-mch}(L)}x_6^{P_6\text{-mch}(L)}.
$$

We use essentially the same ideas as in Section \ref{sec:P1} to decompose  the paths in $\mathcal{L}_n$ to obtain recurrences 
that will allow us to compute $F_{1,6}(x_1,x_6,t)$.
In this case, we take three cases into account. Case 1 are the paths that have no $P_1$-match or $P_6$-match. In addition, we can see paths avoiding $P_1$ and $P_6$ must stay between $y=x-1$ and $y=x+1$. It is easy to see that if the word of such 
as path is $u_1 \ldots u_{2n}$, then either $u_{2i-1}u_{2i} = EN$ or $u_{2i-1}u_{2i} = NE$ for all $i$.
Thus the number of paths in $\mathcal L_n$ bounded by $y=x-1$ and $y=x+1$ is $2^n$. Case 2 are the paths $L$ such 
that the first pattern matching of either $P_1$ or $P_6$ in path $L$ is $P_1$ and  Case 3 are the paths $L$ such that the first pattern matching of either $P_1$ or $P_6$ in $L$ is $P_6$.
Then we have
\begin{eqnarray*}
	F_{1,6}(x_1,x_6,t)&=&\sum_{n\geq 0}2^nt^n+\sum_{i\geq 1}\sum_{j\geq 1}\left(C_i2^{j-1}x_1^it^{i+j}+C_iC2^{j-1}x_6^it^{i+j}\right)F_{1,6}(x_1,x_6,t)\\
	&=&\frac{1}{1-2t}+\frac{t}{1-2t}(C(x_1t)+C(x_6t)-2)F_{1,6}(x_1,x_6,t).
\end{eqnarray*}
Then solving above equation for $F_{1,6}$, we have
\begin{eqnarray*}
	F_{1,6}(x_1,x_6,t)&=&\frac{2x_1x_6}{\left(-1+\sqrt{1-4x_1 t}\right)x_6+\left(-1+\sqrt{1-4x_6 t}\right)x_1+2x_1x_6}\\
	&=&1+2t+(x_1+x_6+4)t^2+(2x_1^2+4x_1+2x_6^2+4x_6+8)t^3\\
	&&+(5x_1^3+9x_1^2+12x_1+5x_6^3+9x_6^2+12x_6+2x_1x_6+16)t^4+\cdots.
\end{eqnarray*}
Clearly, $F_{1,6}(x,1,t)=F_{1,6}(1,x,t)=F_{1}(x,t)$. Next, we discuss coefficients of $x_1t^n$ and $x_1x_6t^n$ 
in $F_{1,6}(x_1,x_6,t)$ which count the number of paths in $\mathcal L_n$ having exactly one $P_1$ pattern and avoiding $P_6$ and the number of paths in $\mathcal L_n$ having exactly one $P_1$ and exactly one $P_6$. 
In general, the generating function for coefficients of $x_1^jx_6^k$ is
\begin{equation}\label{obtain coef}
\left.\frac{1}{j!k!}\frac{\partial^{j+k} F_{1,6}(x_1,x_6,t)}{\partial x_1^j\partial x_6^k}\right|_{x_1=0,x_6=0},
\end{equation}
where if the derivative cannot be evaluated at zero, we take the limit.

By the symmetry of $P_1$ and $P_6$, the coefficient of $x_1t^n$ in $F_{1,6}(x_1,x_6,t)$ equals the coefficient of $x_6t^n$ 
in $F_{1,6}(x_1,x_6,t)$. By Equation (\ref{obtain coef}), the generating function for the coefficients of $x_1t^n$ in $F_{1,6}(x_1,x_6,t)$
equals 
\begin{eqnarray*}
\frac{t^2}{(1-2t)^2}=t^2+4t^3+12t^4+32t^5+80t^6+192t^7+448t^8\cdots.
\end{eqnarray*}
The sequence $1, 4, 12, 32, 80, 192,\cdots$ is A001787 in the OEIS \cite{Slo}. The $n^{th}$ term of this sequence is $n2^{n-1}$ which 
means that the number of paths $L \in \mathcal{L}_n$ with exactly one east step below the subdiagonal $y=x-1$ and 
no east step above the superdiagonal $y=x+1$ equals $(n-1)2^{n-2}$. 
This is easy to prove directly.  That is, if $L$ is such a path, the one east that occurs below the subdiagonal $y =x-1$ 
must arise by starting at a point $[a,a]$ on the diagonal where $0 \leq a \leq n-2$ followed by a sequence $EENN$. If 
we remove this sequence from the word of $L$, we will end up with the word $u_1 \ldots u_{2n-4}$ of path $L' \in \mathcal{L}_{n-2}$ which 
has no east steps either below the subdiagonal $y = x-1$ or above the superdiagonal $y =x+1$.  It is easy to see that in 
such a path $L'$ either $u_{2i-1}u_{2i} = EN$ or $u_{2i-1}u_{2i} = NE$ for all $i$. Hence there are $2^{n-2}$ such paths $L'$ so 
that the number of paths $L \in \mathcal{L}_n$ with exactly one east step below the subdiagonal $y=x-1$ and 
no east step above the superdiagonal $y=x+1$ equals $(n-1)2^{n-2}$.

The generating function of the coefficients of $x_1x_6t^n$ in $F_{1,6}(x_1,x_6,t)$ equals 
\begin{eqnarray*}
\frac{2t^4}{(1-2t)^3}=2t^4+12t^5+48t^6+160t^7+480t^8+1344t^9 + \cdots.
\end{eqnarray*}
The sequence $2,12,48,160,480, \cdots$ is sequence A001815 in the OEIS \cite{Slo}.  We can show directly that 
the number of paths $L \in \mathcal{L}_n$ with exactly one east step below the subdiagonal $y=x-1$ and 
exactly step above the superdiagonal $y=x+1$ equals $(n-2)(n-3) 2^{n-4}$. 
That is, if $L$ is such a path, then the  one east that occurs below the subdiagonal $y =x-1$ 
must arise by starting at a point $[a,a]$ on the diagonal where $0 \leq a \leq n-2$ followed by a sequence $EENN$ 
and the  one east that occurs above  the subdiagonal $y =x+1$ 
must arise by starting at a point $[b,b]$ on the diagonal where $0 \leq a \leq n-2$ followed by a sequence $NNEE$. 
We have $n-1$ choices for the point $[a,a]$. But these $n-1$ choices lead to different situations according to different values of $a$. If $a=0$ or $a=n-2$, we have $n-3$ choices to choose a point $[b,b]$ on the diagonal followed by a sequence $NNEE$. If $0<a<n-2$, there are  $n-4$ choices to choose a point $[b,b]$ on the diagonal followed by a sequence $NNEE$. So the total ways to choose positions of one $P_1$-match and one $P_6$-match is equal to $2(n-3)+(n-3)(n-4)=(n-2)(n-3)$.
We remove  sequence $EENN$ and $NNEE$ from the word of $L$, we will end up with the word $u_1 \ldots u_{2n-8}$ of path $L' \in \mathcal{L}_{n-4}$
which has no east steps either below the subdiagonal $y = x-1$ or above the superdiagonal $y =x+1$.  Hence there are $2^{n-4}$ such paths 
$L'$ so 
that the number of path $L \in \mathcal{L}_n$ with exactly one east step below the subdiagonal $y=x-1$ and 
exactly on east step above the superdiaganal $y=x+1$ equals $(n-2)(n-3) 2^{n-4}$.

If we  are interested in counting lattice paths by the number of east steps below $y=x-1$ or above $y=x+1$, 
then we consider the following generating function: 
\begin{equation*}
	F_{1,6}(x,x,t)=\frac{x}{-1+x+\sqrt{1-4xt}}.
\end{equation*}
And also clearly,
$$
\left.\frac{\partial F_{1,6}(x,x,t)}{\partial x}\right|_{x=1}=2
\left.\frac{\partial F_{1}(x,t)}{\partial x}\right|_{x=1}
$$
because the symmetry of $P_1$ and $P_6$. Then by Equation (\ref{eq:AsyP1}), 
$$
\mathbb{E}[\{P_1,P_6\}\text{-mch}(L):L\in\mathcal L_n]
=2\mathbb{E}[P_1\text{-mch}(L):L\in\mathcal L_n]\sim n+1-2\sqrt{\pi n}.
$$

Next, by manipulating $F_{1,6}(x_1,x_6,t)$ we can also find the number of paths having even number many east steps below $y=x-1$ or above $y=x+1$. The generating function equals 
\begin{eqnarray}
	&&\frac{1}{2}\left(F_{1,6}(1,1,t)+F_{1,6}(-1,-1,t)\right)\\
	&=&1+2t+4t^2+12t^3+36t^4+132t^5+456t^6+1752t ^7\cdots. \label{eq:P16even}
\end{eqnarray}

Similarly, the generating function for the number of paths having odd number many east steps below the subdiagonal $y=x-1$ or above $y=x+1$ is 
\begin{eqnarray}
	&&\frac{1}{2}\left(F_{1,6}(1,1,t)-F_{1,6}(-1,-1,t)\right)\\
	&=&2t^2+8t^3+34t^4+120t^5+468t^6+1680t^7+6530t^8\cdots. \label{eq:P16odd}
\end{eqnarray}
Neither of the two sequences above is recorded in the OEIS \cite{Slo}.

\subsection{$P_2$ and $P_5$}\label{sec:P2P5}
In this subsection, we  shall study
$$
F_{2,5}(x_2,x_5,t):=1+\sum_{n\geq 1}t^n\sum_{L\in\mathcal L_n}x_2^{P_2\text{-mch}(L)}x_5^{P_5\text{-mch}(L)}.
$$
\fig{1.3}{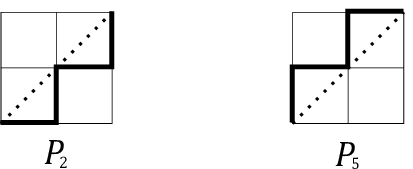}{Pattern $P_2$ and $P_5$.}
$F_{2,5}(x_2,x_5,t)$ is the generating function 
which keeps track of the number of lattice paths by the number of times it bounces off the diagonal to the right or to the left. By 
the symmetry induced by reflecting paths about the diagonal discussed in the introduction, it is easy 
to see that $F_{2,5}(x_2,x_5,t)$ is a symmetric function in $x_2$ and $x_5$. It is also clear that $F_{2,5}(x,x,t)$ 
is the generating function which counts number of times a lattice path in $\mathcal{L}_n$ bounces off the diagonal $y=x$.

First we define 
\begin{equation*}
G_{2,5}(x_2,x_5,t):=\sum_{n\geq 1}t^n\sum_{L\in \mathcal L_n\text{ starting with }E}x_2^{P_2\text{-mch}(L)}x_5^{P_5\text{-mch}(L)}t^n
\end{equation*}
and
\begin{equation*}
H_{2,5}(x_2,x_5,t):=\sum_{n\geq 1}t^n\sum_{L\in \mathcal L_n\text{ starting with }N}x_2^{P_2\text{-mch}(L)}x_5^{P_5\text{-mch}(L)}t^n.
\end{equation*}
Clearly,
$$
F_{2,5}(x_2,x_5,t)=1+G_{2,5}(x_2,x_5,t)+H_{2,5}(x_2,x_5,t).
$$
Here we employ the decomposition of paths used in Section \ref{sec:P2}, then we have
\begin{eqnarray*}
G_{2,5}(x_2,x_5,t)&=&\sum_{i\geq 0}\sum_{j\geq 1}C_{i,j}x_2^it^j\left(H_{2,5}(x_2,x_5,t)+1\right)\\
&=&(C(x_2,t)-1)(H_{2,5}(x_2,x_5,t)+1)
\end{eqnarray*}
and
\begin{eqnarray*}
	H_{2,5}(x_2,x_5,t)&=&\sum_{i\geq 0}\sum_{j\geq 1}C_{i,j}x_5^it^j\left(G_{2,5}(x_2,x_5,t)+1\right)\\
	&=&(C(x_5,t)-1)(G_{2,5}(x_2,x_5,t)+1),
\end{eqnarray*}
where $C(x,t)$ is given as Equation (\ref{eq:C(x,t)}).  
Then
$$
G_{2,5}(x_2,x_5,t)=\left(C(x_2,t)-1\right)\left(\left(C(x_5,t)-1\right)(G_{2,5}(x_2,x_5,t)+1)+1\right).
$$
Solving the above formula for $G_{2,5}$ we have,
\begin{eqnarray*}
	G_{2,5}(x_2,x_5,t)&=&-\frac{(C(x_2,t)-1)C(x_5,t)}{C(x_2,t)(C(x_5,t)-1)-C(x_5,t)}\\
	&=&-\frac{2(1-x_5)t+(x_5-2)\left(1-\sqrt{1-4t}\right)}{1+\sqrt{1-4t}+x_2(x_5-1)\left(1-\sqrt{1-4t}\right)-x_5+\sqrt{1-4t}x_5+2(1-x_2x_5)t}.
\end{eqnarray*}
Therefore,
\begin{eqnarray*}
	F_{2,5}(x_2,x_5,t)&=&1+G_{2,5}(x_2,x_5,t)+H_{2,5}(x_2,x_5,t)\\
	&=&1+G_{2,5}(x_2,x_5,t)+(C(x_5,t)-1)(G_{2,5}(x_2,x_5,t)+1)\\
	&=&C(x_5,t)(G_{2,5}(x_2,x_5,t)+1)\\
	&=&\left(1+\frac{1-\sqrt{1-4t}}{2-x_5\left(1-\sqrt{1-4t}\right)}\right)\cdot \\
	&&\left(1-\frac{2(1-x_5)t+(x_5-2)\left(1-\sqrt{1-4t}\right)}{1+\sqrt{1-4t}+x_2(x_5-1)\left(1-\sqrt{1-4t}\right)-x_5+\sqrt{1-4t}x_5+2(1-x_2x_5)t}\right).
\end{eqnarray*}
A few initial terms of $F_{2,5}(x_2,x_5,t)$ are
\begin{eqnarray*}
	&&F_{2,5}(x_2,x_5,t)\\
	&=&1+2t+(x_2+x_5+4)t^2+(x_2^2+4x_2+x_5^2+4x_5+10)t^3\\
	&&+(x_2^3+5x_2^2+14x_2+x_5^3+5x_5^2+14x_5+2x_2x_5+28)t^4\\
	&&+	(x_2^4+6x_2^3+21x_2^2+48x_2+x_5^4+6x_5^3+21x_5^2+48x_5+2x_2^2x_5+2x_2x_5^2+12x_2x_5+84)t^5\\
	&&+\cdots.
\end{eqnarray*}
By Equation (\ref{obtain coef}), we can obtain the generating functions of the coefficients of $x_2t^n$ in $F_{2,5}(x_2,x_5,t)$ 
which equals 
\begin{eqnarray*}
	\left.\frac{\partial F_{2,5}(x_2,0,t)}{\partial x_2}\right|_{x_2=0}&=&
	\frac{1-\sqrt{1-4t}+2t(-2+\sqrt{1-4t}+t)}{2t^2}\\
	&=&t^2+4t^3+14t^4+48t^5+165t^6+572t^7+7072t^8+\cdots\\
	&=&	\left.\frac{\partial F_{1}(x,t)}{\partial x}\right|_{x\rightarrow 0}.
\end{eqnarray*}
This implies there exists a bijection between paths having exactly one $P_2$-match but no $P_5$-matches and 
paths having exactly one step below $y=x-1$. We leave this as an open problem. 

Similarly, we can get coefficients of $x_2x_5t^n$,
\begin{eqnarray*}
	\left.\frac{\partial^2 F_{2,5}(x_2,x_5,t)}{\partial x_2\partial x_5}\right|_{x_2=x_5=0}
	&=&2t^4+12t^5+56t^6+236t^7+948t^8+3712t^9\cdots.
\end{eqnarray*}
The sequence $2,12,56,236,948,\cdots$ is not in the OEIS \cite{Slo}.

It is also the case that $F_{2,5}(1,x,t)=F_{5}(x,t)=F_2(x,t)=F_{2,5}(x,1,t)$ and
\begin{eqnarray*}
	F_{2,5}(0,0,t)&=&1+2t+4t^2+10t^3+28t^4+84t^5+264t^6\cdots\\
	&=&1+2C_1t+2C_2t^2+2C_3t^3+2C_4t^4+2C_5t^5+\cdots,
\end{eqnarray*}
where $C_k$ is the $k^{th}$ Catalan number.
$F_{2,5}(x,x,t)$ is the generating function over paths $L$ in $\mathcal{L}_n$ by the number of times $L$ bounces off the diagonal.

\begin{eqnarray*}
	F_{2,5}(x,x,t)&=&\frac{1-\sqrt{1-4t}-t-x+x^2t}{-x+(1+x^2)t}\\
	&=&1+2t+2(x+2)t+2(x+2)t^2+2(x^2+4x+5)t^3\\
	&&+2(x^3+6x^2+14x+14)t^4+2(x^4+8x^3+27x^2+48x+42)t^5\cdots.
\end{eqnarray*}
We take partial derivative of $F_{2,5}(x,x,t)$ with respect $x$ and evaluate at $x=1$,
\begin{eqnarray*}
	\left.\frac{\partial F_{2,5}(x,x,t)}{\partial x}\right|_{x=1}&=&\frac{\sqrt{1-4t}}{-1+4t}-\frac{1-2t}{-1+4t}\\
	&=&\sum_{n\geq 2}\left(\frac{4^n}{2}-2\binom{2n-1}{n-1}\right)t^n\\
	&=&2t^2+12t^3+58t^4+260t^5+1124t^6+4760t^6+19898t^8+\cdots\\
	&=&2\left.\frac{\partial F_{2}(x,t)}{\partial x}\right|_{x=1}.
\end{eqnarray*}
It follows that 
$$
\mathbb{E}[\{P_2,P_5\}\text{-mch}(L):L \in \mathcal{L}_n]=2\mathbb{E}[P_2\text{-mch}(L):L \in \mathcal{L}_n]\approx \frac{\sqrt{\pi n}}{2}-1\approx 0.886\sqrt{n}
$$
gives the expected number of times a path in $\mathcal L_n$ bounces off the diagonal.

$\frac{1}{2}(F_{2,5}(1,1,t)+F_{2,5}(-1,-1,t))$ is the generating function of 
the number of lattice paths in $\mathcal{L}_n$ that bounce off the diagonal an even number of times. We have computed that 
\begin{eqnarray*}
	&&\frac{1}{2}\left(F_{2,5}(1,1,t)+F_{2,5}(-1,-1,t)\right)\\
	&=&\frac{1-\sqrt{1-4t}+\left(-6+4\sqrt{1-4t}\right)t+4t^2}{1-\sqrt{1-4t}+\left(-4+2\sqrt{1-4t}\right)t}\\
	&=&1+2\sum_{n\geq 1}\binom{2n-2}{n-1}t^n\\
	&=&1+2t+4t^2+12t^3+40t^4+140t^5+504t^6+\cdots.
\end{eqnarray*}
The sequence $2,4,12,40,140,\cdots$ is sequence A028329 in the OEIS \cite{Slo}. 
It would 
be nice to have a direct combinatorial proof that the 
number of lattice paths in $\mathcal{L}_n$ that bounce off the diagonal an even number of times equals $2\binom{2n-2}{n-1}$.

$\frac{1}{2}(F_{2,5}(1,1,t)-F_{2,5}(-1,-1,t))$ is the generating function of the 
number of lattice paths $L$ in $\mathcal{L}_n$  that bounce off the diagonal an odd number of times. We have computed that 
\begin{eqnarray*}
	&&\frac{1}{2}\left(F_{2,5}(1,1,t)-F_{2,5}(-1,-1,t)\right)\\
	&=&2\sum_{n\geq 2}\binom{2n-2}{n-2}t^n\\
	&=&2t^2+8t^3+30t^4+112t^5+420t^6+1584t^7\cdots.
\end{eqnarray*}
The sequence $2, 8, 30, 112, 420, 1584,\cdots$ is sequence A162551 in the OEIS \cite{Slo}.
It would 
be nice to have a direct combinatorial proof of that the 
number of lattice paths in $\mathcal{L}_n$ that bounce off the diagonal an odd number of times equal $2\binom{2n-2}{n-2}$.

\subsection{$P_3$ and $P_4$}
We define
\begin{equation}
F_{3,4}(x_3,x_4,t):=1+\sum_{n\geq 1}t^n\sum_{L\in \mathcal L_n}x_3^{P_3\text{-mch}(L)}x_4^{P_4\text{-mch}(L)},
\end{equation}
where $x_3$ is used to keep track of the number of horizontal crossings and $x_4$ is used to keep track of the number of vertical crossings. Clearly, $F_{3,4}(x_3,x_4,t)$ is symmetric in $x_3$ and $x_4$.

\fig{1.3}{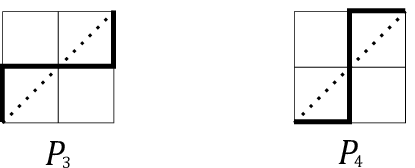}{Pattern $P_3$ and $P_4$.} 
We also define 
\begin{equation*}
G_{3,4}(x_3,x_4,t):=\sum_{n\geq 1}t^n\sum_{L\in \mathcal L_n\text{ starting with }E}x_3^{P_3\text{-mch}(L)}x_4^{P_4\text{-mch}(L)}
\end{equation*}
and
\begin{equation*}
H_{3,4}(x_3,x_4,t):=\sum_{n\geq 1}t^n\sum_{L\in \mathcal L_n\text{ starting with }N}x_3^{P_3\text{-mch}(L)}x_4^{P_4\text{-mch}(L)}.
\end{equation*}
We employ the same decomposition of paths used in Section \ref{sec:P2P5} for $P_3$ and $P_4$. Then
\begin{equation*}
H_{3,4}(x_3,x_4,t)=\sum_{j\geq 1}C_jt^j\left(x_3G_{3,4}(x_3,x_4,t)+1\right)=(C(t)-1)(x_3G_{3,4}(x_3,x_4,t)+1)
\end{equation*}
and
\begin{equation*}
G_{3,4}(x_3,x_4,t)=\sum_{j\geq 1}C_{j}t^j\left(x_4H_{3,4}(x_3,x_4,t)+1\right)=(C(t)-1)(x_4H_{3,4}(x_3,x_4,t)+1).
\end{equation*}
Combining the two equations above, we can then solve for $G_{3,4}$ to obtain that 
\begin{eqnarray*}
G_{3,4}(x_3,x_4,t)&=&\frac{(1-C(t))((x_4C(t)-1)+1)}{x_3x_4(C(t)-1)^2-1}\\
&=&-\frac{\left(-1+\sqrt{1-4t}+2t\right)\left(2t(1-x_4)+\left(-1+\sqrt{1-4t}\right)x_4\right)}{-2(-1+\sqrt{1-4t})x_3x_4+4(-2+\sqrt{1-4t})x_3x_4t+4(x_3x_4-1)t^2}.
\end{eqnarray*}
Then
\begin{eqnarray*}
F_{3,4}(x_3,x_4,t)
&=&1+G_{3,4}(x_3,x_4,t)+H_{3,4}(x_3,x_4,t)\\
&=&1+G_{3,4}(x_3,x_4,t)+(C(t)-1)(x_3G_{3,4}(x_3,x_4,t)+1)\\
&=&\left(x_3C(t)-x_3+1\right)G_{3,4}(x_3,x_4,t)+C(t)\\
&=&\frac{1-\sqrt{1-4t}}{2t}-\frac{\left(1-\frac{1-\sqrt{1-4t}}{2t}\right)\left(1-x_3+\frac{1-\sqrt{1-4t}}{2t}x_3\right)
	\left(1-x_4+\frac{1-\sqrt{1-4t}}{2t}x_4\right)}{-1+\left(-1+\frac{1-\sqrt{1-4t}}{2t}\right)^2x_3x_4}.
\end{eqnarray*}
A few initial terms of $F_{3,4}(x_3,x_4,t)$ are
\begin{eqnarray*}
	&&F_{3,4}(x_3,x_4,t)\\
	&=&1+2t+(x_3+x_4+4)t^2+(4x_3+4x_4+2x_3x_4+10)t^3\\
	&&+(14x_3+14x_4+x_3^2x_4+x_3x_4^2+12x_3x_4+28)t^4\\
	&&+(48x_3+48x_4+2x_3^2x_4^2+8x_3^2x_4+8x_3x_4^2+54x_3x_4+84)t^5\\
	&&+\cdots.
\end{eqnarray*}
By symmetry, $F_{3,4}(1,x,t)=F_{4}(x,t)=F_3(x,t)=F_{3,4}(x,1,t)$. It is also clear that 
$F_{3,4}(0,0,t)=F_{2,5}(0,0,t)=2C(t)$, where $C(t)$ is the generating function of Catalan numbers, since if a path in 
$\mathcal{L}_n$ 
has no vertical or horizontal crossings, then the path either stays on or below the diagonal or on and above 
the diagonal. 

By Equation (\ref{obtain coef}), we see that coefficients of $x_3t^n$ in $F_{3,4}(x_3,x_4,t)$ 
yield the generating function of the number of paths in $\mathcal L_n$ that have exactly one horizontal crossing and no vertical 
crossings. We have computed that 
\begin{eqnarray*}
	\left.\frac{\partial F_{3,4}(x_3,0,t)}{\partial x_3}\right|_{x_3=0}&=&
	\frac{1-\sqrt{1-4t}+2t(-2+\sqrt{1-4t}+t)}{2t^2}\\
	&=&t^2+4t^3+14t^4+48t^5+165t^6+572t^7+7072t^8+\cdots\\
	&=&	\left.\frac{\partial F_{1}(x,t)}{\partial x}\right|_{x\rightarrow 0}\\
	&=&\left.\frac{\partial F_{2,5}(x_2,0,t)}{\partial x_2}\right|_{x_2=0},
\end{eqnarray*}
which implies  the number of paths in $\mathcal L_n$ having exactly one $P_3$-match and avoiding $P_4$ is equal to the number   of paths in $\mathcal L_n$ having exactly one $P_2$-match and avoiding $P_5$. This can be verified by the bijection defined in Section \ref{sec:bije}. However, coefficients of $x_3x_4t^n$ in $F_{3,4}(x_3,x_4,t)$ is not equal to the coefficient of $x_2x_5t^n$ in $F_{2,5}(x_2,x_5,t)$. 
This is due to the fact that  a path in $\mathcal{L}_n$ cannot cross the diagonal horizontally twice without crossing the diagonal vertically.  We have computed that 
\begin{eqnarray*}
	\left.\frac{\partial^2 F_{3,4}(x_3,x_4,t)}{\partial x_3\partial x_4}\right|_{x_3=x_4=0}
		&=&2\left.\frac{\partial F_{3}(x,t)}{\partial x}\right|_{x=0}\\
	&=&-\frac{8t^2\left(-1+\sqrt{1-4t}+2t\right)}{(\sqrt{1-4t}\left(1+\sqrt{1-4t}-2t\right)^3}\\
	&=& 2t^2+12t^3+54t^4+220t^5+858t^6+3276t^7+\cdots,
\end{eqnarray*}
The sequence $2, 12, 54, 220, 858, 3276,\cdots$ is Column 2 in A118920 and the exactly same interpretation is given by 	
Emeric Deutsch in the OEIS \cite{Slo}.

For $F_{{3,4}}(x,x,t)$, we can show that $F_{3,4}(x,x,t)=F_{2,5}(x,x,t)$  by the bijection defined in Section \ref{sec:bije}, which gives us that for a random $L\in\mathcal L_n$, the expectation of the number of crossings has asymptotic approximation as follows,
$$
\mathbb{E}[\{P_3,P_4\}\text{-mch}(L)]\sim\frac{\sqrt{\pi n}}{2}-1\approx 0.886\sqrt{n},
$$
and also
$$
\frac{1}{2}(F_{3,4}(1,1,t)+F_{3,4}(-1,-1,t))=
\frac{1}{2}(F_{2,5}(1,1,t)+F_{2,5}(-1,-1,t)).
$$

\subsection{$P_2$ and $P_4$}
Due to space limitations, we shall consider only one more set of patterns of size 2, namely, $\Delta = \{2,4\}$. 
First, we define
$$
F_{2,4}(x_2,x_4,t):=1+\sum_{n\geq 1}t^n\sum_{L\in\mathcal L_n}x_2^{P_2\text{-mch}(L)}x_4^{P_4\text{-mch}(L)}.
$$
\fig{1.3}{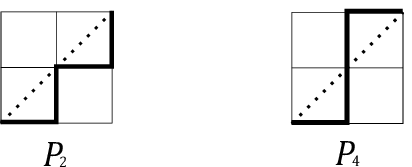}{Pattern $P_2$ and $P_4$.}
$F_{2,4}(x_2,x_4,t)$ counts the number of lattice paths by the number of times it bounces off the diagonal to the right and by the 
the number of times it crosses the diagonal vertically. It follows that $F_{2,4}(x,x,t)$ is the generating function 
over lattice paths $L$ in $\mathcal{L}_n$ by the number of times $L$ touches the diagonal with a north step.  By symmetry, 
 $F_{3,5}(x,x,t)$ is also the generating function 
over lattice paths $L$ in $\mathcal{L}_n$ by the number of times $L$ touches the diagonal with an east step. 

First we define 
\begin{equation*}
G_{2,4}(x_2,x_4,t):=\sum_{n\geq 1}t^n\sum_{L\in \mathcal L_n\text{ starting with }E}x_2^{P_2\text{-mch}(L)}x_4^{P_4\text{-mch}(L)}
\end{equation*}
and
\begin{equation*}
H_{2,4}(x_2,x_4,t):=\sum_{n\geq 1}t^n\sum_{L\in \mathcal L_n\text{ starting with }N}x_2^{P_2\text{-mch}(L)}x_4^{P_4\text{-mch}(L)}.
\end{equation*}
Clearly,
$$
F_{2,4}(x_2,x_4,t)=1+G_{2,4}(x_2,x_4,t)+H_{2,4}(x_2,x_4,t).
$$
Employing the same decomposition that is used in Section \ref{sec:P2P5}, we have
\begin{eqnarray*}
G_{2,4}(x_2,x_4,t)&=&\sum_{i\geq 0}\sum_{j\geq 1}C_{i,j}x_2^it^j\left(x_4H_{2,4}(x_2,x_4,t)+1\right)\\
&=&(C(x_2,t)-1)\left(x_4H_{2,4}(x_2,x_4,t)+1\right)
\end{eqnarray*}
and
\begin{eqnarray*}
	H_{2,4}(x_2,x_4,t)&=&\sum_{j\geq 1}C_jt^j(G_{2,4}(x_2,x_4,t)+1)\\
	&=&(C(t)-1)(G_{2,4}(x_2,x_4,t)+1).
\end{eqnarray*}
Combining the two equations above, we can solve them for $G_{2,4}$,
\begin{eqnarray*}
	G_{2,4}(x_2,x_4,t)&=&-\frac{(C(x_2,t)-1)(x_4(C(t)-1)+1)}{x_4(C(x_2,t)-1)(C(t)-1)-1}.
\end{eqnarray*}
Then
\begin{eqnarray*}
	F_{2,4}(x_2,x_4,t)&=&1+G_{2,4}(x_2,x_4,t)+H_{2,4}(x_2,x_4,t)\\
	&=&1+G_{2,4}(x_2,x_4,t)+(C(t)-1)(G_{2,4}(x_2,x_4,t)+1)\\
	&=&C(t)(G_{2,4}(x_2,x_4,t)+1)\\
	&=&\frac{(x_2-2)\left(-1+\sqrt{1-4t}\right)+2(x_2-1)t}{x_4\left(-1+\sqrt{1-4t}\right)+x_2\left(2+\left(-1+\sqrt{1-4t}\right)+3x_4-x_4\sqrt{1-4t}\right)t}.
\end{eqnarray*}
A few initial terms are
\begin{eqnarray*}
	&&F_{2,4}(x_2,x_4,t)\\
	&=&1+2t+(x_2+x_4+4)t^2+(x_2^2+3x_2+5x_4+x_2x_4+10)t^3\\
	&&+(x_2^3+4x_2^2+9x_2+x_4^2+19x_4+x_2^2x_4+7x_2x_4+28)t^4\\
	&&+(x_2^4+5x_2^3+14x_2^2+28x_2+8x_4^2+68x_4+x_2^3x_4+2x_2x_4^2+9x_2^2x_4+32x_2x_4+84)t^5+\cdots.
\end{eqnarray*}

By Equation (\ref{obtain coef}), the coefficient of $x_2t^n$ in $F_{2,4}(x_2,x_4,t)$ is 
the number of paths in $\mathcal{L}_n$ which bounce off the diagonal to the right exactly one time but do not cross the diagonal vertically. We have 
computed that 
\begin{eqnarray*}
	\left.\frac{\partial F_{2,4}(x_2,0,t)}{\partial x_2}\right|_{x_2=0}&=&
	-\frac{\left(-1+\sqrt{1-4t}\right)^3}{8t}\\
	&=&t^2+3t^3+9t^4+28t^5+90t^6+297t^7+1001t^8+\cdots.
\end{eqnarray*}
The sequence $1,3,9,28,90,297,\cdots$ is  sequence A000245 in the OEIS \cite{Slo} which has  several interpretations such as the number of permutations on $\{1,2,\cdots,n+2\}$ that are 123-avoiding and for which the integer $n$ is in the third spot, the number of lattice paths in $\mathcal L_{n-1}$ which touch but do not cross the $y=x-1$ and the number of Dyck paths in $\mathcal L_{n}$ that start with `$EE$'.

Similarly, the coefficient of $x_4t^n$ in $F_{2,4}(x_2,x_4,t)$ is the number of paths 
in $\mathcal{L}_n$ which have exactly one vertical crossing but never bounce off the diagonal to the right. We have computed that 
\begin{eqnarray*}
	\left.\frac{\partial F_{2,4}(0,x_4,t)}{\partial x_4}\right|_{x_4=0}&=&
	-\frac{\left(-3+\sqrt{1-4t}\right)\left(-1+\sqrt{1-4t}+2t\right)^2}{8t^2}\\
	&=&t^2+5t^3+19t^4+68t^5+240t^6+847t^7+3003t^8+\cdots. 
\end{eqnarray*}
The sequence 
$1,5,19,68,240,\cdots$ is sequence A070857 in the OEIS \cite{Slo} which has no combinatorial interpretation. Thus we have given 
a combinatorial interpretation to this sequence. 

The coefficient of $x_2x_4t^n$ in $F_{2,4}(x_2,x_4,t)$ is the number of paths in $\mathcal{L}_n$ which bounce off the diagonal 
to the right exactly once and cross the diagonal vertically exactly one.  The corresponding generating function equals 
\begin{eqnarray*}
	\left.\frac{\partial^2 F_{2,4}(x_2,x_4,t)}{\partial x_2 \partial x_4}\right|_{x_2=x_4=0}&=&
	-\frac{\left(-1+\sqrt{1-4t}\right)^3\left(-2+\sqrt{1-4t}\right)\left(-1+\sqrt{1-4t}+2t\right)}{16t^2}\\
	&=&t^3+7t^4+32t^5+129t^6+495t^7+1859t^8+\cdots,
\end{eqnarray*}
which has no matches in the OEIS \cite{Slo}.

As we mentioned, $F_{2,4}(x,x,t)$ is the generating function for the times of paths touching the diagonal $y=x$ with a north step,
\begin{eqnarray*}
	F_{2,4}(x,x,t)&=&\frac{1-\sqrt{1-4t}-t-x+x^2t}{-x+(1+x)^2t}\\
	&=&1+2t+(2x+4)t^2+(2x^2+8x+10)t^3+(2x^3+12x^2+28x+28)t^4+\cdots\\
	&=& F_{2,5}(x,x,t)=F_{3,4}(x,x,t).
\end{eqnarray*}

This fact can be also shown by constructing a bijection. Let $\mathcal C_{n,k}$ denote the set of paths in $\mathcal L_{n}$ that cross the diagonal $k$ times and $\mathcal T_{n,k}$ denote the set of paths in $\mathcal L_{n}$ that touch the diagonal with a north step $k$ times. 

Next, we shall construct a bijection between $\mathcal T_{n,k}$ and $\mathcal{C}_{n,k}$, which is similar to the bijection defined in Section \ref{sec:bije}. For any path $L\in\mathcal T_{n,k}$, assume $L$ touches the diagonal $j$ times and these positions are denoted by $\{p_1,p_2,\cdots,p_j\}$. We let $p_0=[0,0]$ and $p_{j+1}=[n,n]$. If $p_i$ is a bouncing right position or $p_{i}$ is a horizontal crossing position, we flip the part between $p_{i-1}$ and $p_i$ along the diagonal. Then we obtain a new path $L'$. In this bijection, we can see that the number of crossings of $L$ is equal to the number of north-touchings of $L'$, and   the number of north-touchings of $L$ is equal to the number of crossings of $L'$. For example, in Figure \ref{fig:BijCross}, $L$ is mapped to $L'$ and $\{P_2,P_4\}\text{-mch}(L)=\{P_3,P_4\}\text{-mch}(L')=3$ and 
$\{P_3,P_4\}\text{-mch}(L)=\{P_2,P_4\}\text{-mch}(L')=2$.

\fig{1.0}{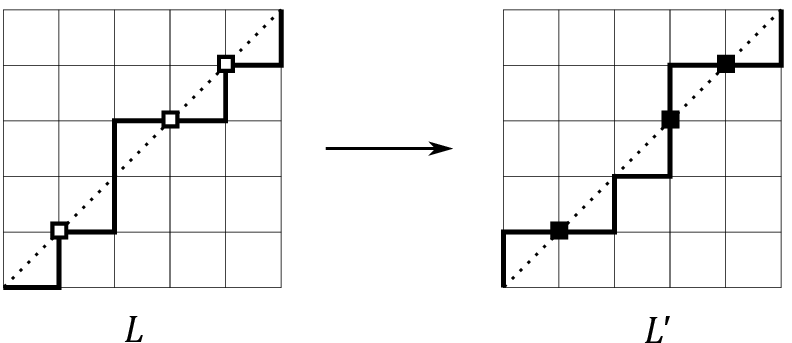}{$L$ is mapped to $L'$ by the bijection.}

Because $F_{2,4}(x,x,t)=F_{3,4}(x,x,t)$, for a random $L\in\mathcal L_n$,
$$
\mathbb{E}[\{P_2,P_4\}\text{-mch}(L)]\sim \frac{\pi n}{2}-1\approx 0.886 \sqrt{n}
$$

\subsection{$P_2$, $P_3$, $P_4$ and $P_5$}
The last example of this section is a generating functions of a subset of $\{1, \ldots, 6\}$ of size $4$. That is, 
we shall study the generating function
$$
F_{2,3,4,5}(x_2,x_3,x_4,x_5,t):=1+\sum_{n\geq 1}t^n\sum_{L\in\mathcal L_n}x_2^{P_2\text{-mch}(L)}x_3^{P_3\text{-mch}(L)}x_4^{P_4\text{-mch}(L)}x_5^{P_5\text{-mch}(L)}.
$$
For convenience, in this subsection we use $F_{2,3,4,5}$ to denote $F_{2,3,4,5}(x_2,x_3,x_4,x_5,t)$,  $G_{2,3,4,5}$ to denote $G_{2,3,4,5}(x_2,x_3,x_4,x_5,t)$ and  $H_{2,3,4,5}$ to denote $H_{2,3,4,5}(x_2,x_3,x_4,x_5,t)$ where 
$$
G_{2,3,4,5}:=1+\sum_{n\geq 1}t^n\sum_{L\in\mathcal L_n\text{ starting with }E}\prod_{k=2}^5x_k^{P_k\text{-mch}(L)}
$$
and
$$
H_{2,3,4,5}:=1+\sum_{n\geq 1}t^n\sum_{L\in\mathcal L_n\text{ starting with }N}\prod_{k=2}^5x_k^{P_k\text{-mch}(L)}.
$$
Similar to the recurrences used in previous subsections, we have
\begin{eqnarray*}
	G_{2,3,4,5}&=&\sum_{i\geq 0}\sum_{j\geq 1}C_{i,j}x_2^it^j\left(x_4H_{2,3,4,5}+1\right)\\
	&=&(C(x_2,t)-1)\left(x_4H_{2,3,4,5}+1\right)
\end{eqnarray*}
and
\begin{eqnarray*}
	H_{2,3,4,5}&=&\sum_{i\geq 0}\sum_{j\geq 1}C_{i,j}x_5^it^j\left(x_3G_{2,3,4,5}+1\right)\\
	&=&(C(x_5,t)-1)\left(x_3G_{2,3,4,5}+1\right)
\end{eqnarray*}
Combining the two equations above, we can solve them for $G_{2,3,4,5}$,
\begin{eqnarray*}
	G_{2,3,4,5}&=&\frac{(C(x_2,t)-1)(x_4(C(x_5,t)-1)+1)}{x_3x_4(C(x_2,t)-1)(C(x_5,t)-1)-1}.
\end{eqnarray*}
Then
\begin{eqnarray*}
	&&F_{2,3,4,5}\\
	&=&1+G_{2,3,4,5}+H_{2,3,4,5}\\
	&=&1+G_{2,3,4,5}+(C(x_5,t)-1)\left(x_3G_{2,3,4,5}+1\right)\\
	&=&C(x_5,t)\left(x_3G_{2,3,4,5}+1\right)+(1-x_3)G_{2,3,4,5}\\
	&=&\frac{P(x_2,x_3,x_4,x_5,t)}{Q(x_2,x_3,x_4,x_5,t)},
\end{eqnarray*}
where
\begin{multline*}
P(x_2,x_3,x_4,x_5,t) = \\
(-1 + \sqrt{1 - 4 t} + 2 t) x_3 (-1 + x_4) + x_4 - \sqrt{1 - 4 t} x_4 -
2 t x_4+ x_2 (-(-1 + \sqrt{1 - 4 t}) (-2 + x_5) -\\
2 t (-1 + x_5)) +
2 \sqrt{1 - 4 t} x_5 + 2 t x_5 -
2 (-2 + \sqrt{1 - 4 t} + x_5)
\end{multline*}
and
\begin{multline*}
Q(x_2,x_3,x_4,x_5,t) = \\
2 + (-1 + \sqrt{1 - 4 t} +
2 t) x_3 x_4 + (-1 + \sqrt{1 - 4 t}) x_5 +
x_2 (-1 + \sqrt{1 - 4 t} - (-1 + \sqrt{1 - 4 t} + 2 t) x_5).
\end{multline*}

One can imagine that even a few initial terms of $F_{2,3,4,5}(x_2,x_3,x_4,x_5,t)$ are very long so that we 
will not list them here. However, it is easy to verify that the constant coefficients of $t^n$ is just $2C_n$ because there are two sets of Dyck paths, namely, 
the ones that stay on or above the diagonal and the ones that stay on or below the diagonal. 

By manipulating $F_{{2,3,4,5}}(x_2,x_3,x_4,x_5,t)$, one is able to answer certain complicated enumerative problems, such as how many paths in $\mathcal L_n$ are there that cross the diagonal vertically exactly once and horizontally exactly twice, and bounce off the diagonal to the right once but not to the left. The answer to this question has the generating function as follows,
\begin{eqnarray*}
&&\frac{1}{2!}\left.\frac{\partial^4 F_{2,3,4,5}(x_2,x_3,x_4,0,t)}{\partial x_2\partial x_3^2\partial x_4}\right|_{x_2=x_3=x_4=0}\\
&=&
\frac{\left(1-\sqrt{1-4t}\right)^5}{16}\\
&=&2t^5+10t^6+40t^7+150t^8+550t^9+2002t^{10}+7280t^{11}\cdots.
\end{eqnarray*}
Amazingly, the sequence $2,10,40,150 550,2002,\cdots$ is twice the  sequence A000344 in the OEIS \cite{Slo}, which has interpretations such as the number of paths in $
\mathcal L_{n-3}$ that touch but do not cross $y=x-2$ and the number of standard tableaux of shape $(n-1,n-5)$. We leave open the problem of finding a bijective proofs of these facts.

Next, we consider the formula $F_{2,3,4,5}(x,x,x,x,t)$ which gives us the generating functions for the times of touching the diagonal,
\begin{eqnarray*}
F_{2,3,4,5}(x,x,x,x,t)&=&\frac{1+(x-1)\left(-1+\sqrt{1-4t}\right)}{1+\left(-1+\sqrt{1-4t}\right)x}\\
&=&1+2t+(4x+2)t^2+(8x^2+8x+4)t^3+(16x^3+24x^2+20x+10)t^4\\
&&+(32x^4+64x^3+72x^2+56x+28)t^5+\cdots.
\end{eqnarray*}

Next, we want to ask for a random $L\in\mathcal L_n$ how many times in average that $L$ touches the diagonal. Applying the same idea that 
we used in previous sections, we see that 
\begin{eqnarray*}
	\left.\frac{\partial F_{2,3,4,5}(x,x,x,x,t)}{\partial x}\right|_{x=1}&=&\frac{\left(\sqrt{1-4t}-1\right)^2}{4t-1}\\
	&=&4t^2+24t^3+116t^4+520t^5+2248t^6+9530t^7+\cdots\\
	&=&4	\left.\frac{\partial F_{2}(x,t)}{\partial x}\right|_{x=1}.
\end{eqnarray*}
So for a random $L\in\mathcal L_n$, the expectation of times $L$ touches the diagonal is that
$$
\mathbb{E}[\{P_2,P_3,P_4,P_5\}\text{-mch}(L)]=\frac{4^n-4\binom{2n-1}{n-1}}{\binom{2n}{n}}\sim \sqrt{\pi n}-2\approx 1.772\sqrt{n}.
$$

Similarly, we can also obtain the generating functions for the number of paths touching the diagonal an even number of times or 
an odd number of times.  We have computed that 
\begin{eqnarray}
&&\frac{1}{2}\left(F_{2,3,4,5}(1,1,1,1,t)+F_{2,3,4,5}(-1,-1,-1,-1,t)\right)\\
&=&\frac{4t+\sqrt{1-4t}}{4t+2\sqrt{1-4t}-1}\\
&=&1+2t+2t^2+12t^3+34t^4+132t^5+468t^6+1752t^7+6530t^8+\cdots. \label{eq:P2345even}
\end{eqnarray}
and 
\begin{eqnarray*}
	&&\frac{1}{2}\left(F_{2,3,4,5}(1,1,1,1,t)-F_{2,3,4,5}(-1,-1,-1,-1,t)\right)\\
	&=&-\frac{2\left(-1+\sqrt{1-4t}+2t\right)}{-1+2\sqrt{1-4t}+4t}\\
&=&4t^2+8t^3+36t^4+120t^5+456t^6+1680t^7+6340t^8+\cdots.
\end{eqnarray*}
Neither of the two series have matches in the OEIS \cite{Slo}. 

By observing Equation (\ref{eq:P16even}) and (\ref{eq:P16odd}), we find that coefficient of $t^k$ in Equation (\ref{eq:P2345even}) is equal to
$$
\begin{cases}
\text{coefficient of $t^k$ in }\frac{1}{2}(F_{1,6}(1,1,t)-F_{1,6}(1,1,t))\text{, if }k\text{ is even}\\
\text{coefficient of $t^k$ in }\frac{1}{2}(F_{1,6}(1,1,t)+F_{1,6}(1,1,t))\text{, if }k\text{ is odd}.
\end{cases}
$$

This is because all the six patterns in $\mathcal L_2$ are mutually exclusive. For any path $L\in\mathcal L_{k}$, $\mathcal L_{2}\text{-mch}(L)=k-1$, which implies that
$$
\{P_1,P_6\}\text{-mch}(L)+\{P_2,P_3,P_4,P_5\}\text{-mch}(L)=k-1.
$$
If $k$ is odd, $\{P_1,P_6\}\text{-mch}(L)$ and $\{P_2,P_3,P_4,P_5\}\text{-mch}(L)$ have the same parity and otherwise, they do not.

\section{Future research}
In this paper, we computed the generating functions $F_{P_k}(x,t)$ for $k =1, \ldots 6$ and $F_{\Delta}(\mathbf{x},t)$ for certain selected $\Delta \subseteq \{1, \ldots,n\}$.  In a subsequent paper, we will systematically compute $F_{\Delta}(\mathbf{x},t)$ for all 
sets of size two. There are only nine such generating functions up to symmetry and we 
have computed five of them since $F_{P_2,P_3}(x_2,x_3,t)$ is a specialization of $F_{2,3,4,5}$. The ones that we have not computed 
in the paper are represented by $F_{P_1,P_2}(x_1,x_2,t)$, $F_{P_1,P_3}(x_1,x_3,t)$, $F_{P_1,P_4}(x_1,x_4,t)$, and 
$F_{P_1,P_5}(x_1,x_5,t)$. As a special case for pattern $P_1,P_2$, $F_{1,2}(x,x,t)$ keeps track of the number of paths in $\mathcal L_n$ that have $k$ steps below the diagonal. For any fixed $k$, the coefficient of $x^kt^n$ in $F_{1,2}(x,x,t)$ is also Catalan number $C_n$, shown by Chung and Feller \cite{CF}. We shall explore these generating functions in a subsequent paper where we will also add 
some additional parameters which keep track of both the area below the diagonal and the area above the diagonal 
in path in $\mathcal{L}_n$. 

There are many interesting bijective problems that arise from our results.  For example, in Section 3.1, we find that the total east steps below $y=x-1$ of all the paths in $\mathcal L_n$ is equal to the total area under all Dyck paths in $\mathcal L_n$. We take $\mathcal L_3$ as an example, there are 6 paths having $P_1$-matches and there are 5 Dyck paths, pictured in Figure \ref{fig:Dyck}. The total east steps below $y=x-1$ is equal to $2+2+1+1+1+1=8$ and the total area under all the Dyck paths  is also equal to $0+1+2+2+3=8$. Although how to design the bijection is unknown, it is interesting to see paired pattern matching does have connection to other statistics for lattice paths.

\fig{1}{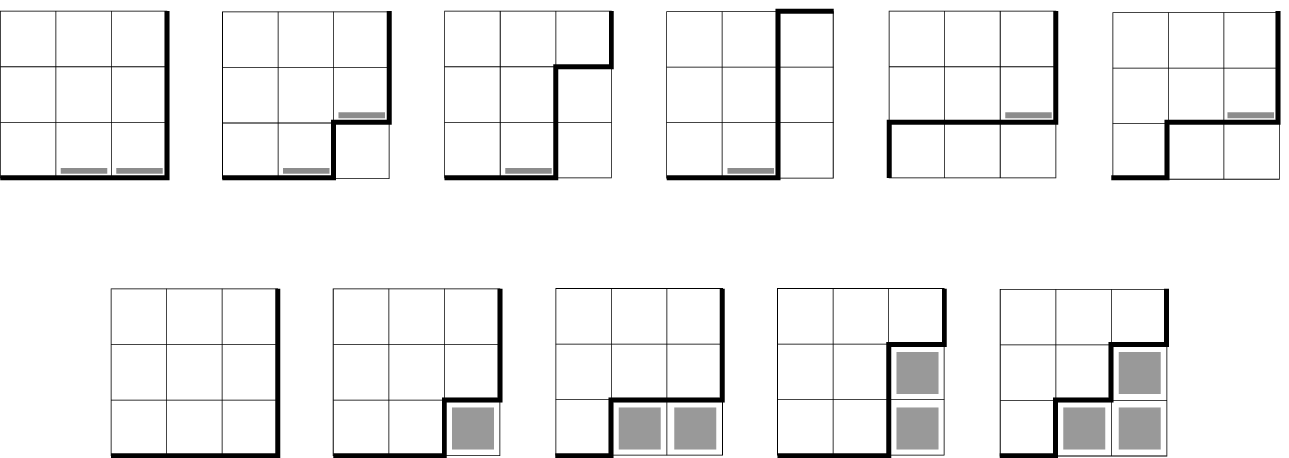}{Total number of east steps below $y=x-1$ in $\mathcal L_n$ equals the total area below Dyck paths in $\mathcal D_n$, $n=3$ as an example.}

Another direction for further research is to consider  Delannoy paths. In this paper, we only consider paths consisting of 
north steps $[0,1]$ and east steps $[1,0]$. Naturally, we can extend our definitions to  Delannoy paths which are paths 
consisting of east steps $[1,0]$, north steps $[0,1]$, and northeast steps $[1,1]$ which start at $[0,0]$ 
and end at $[n,n]$. We denote the steps $[1,0]$, $[0,1]$ and $[1,1]$ by $E$, $N$ and $D$ respectively.
The set of all the   Delannoy paths from $[0,0]$ to $[n,n]$ is denoted by $\mathcal S_n$.

According to \cite{Sch}, a Schr\"{o}der path is a path from $[0,0]$ to $[n,n]$ consisting of east steps  $[1,0]$, north steps  $[0,1]$, and northeast steps $[1,1]$ which never goes above the diagonal $y=x$. The number of Schr\"{o}der paths  from $[0,0]$ to $[n,n]$ is counted by large Schr\"{o}der number $D_n$ whose ordinary generating function equals 
$$
	D(x)=\sum_{n\geq 0}D_nx^n=\frac{1-x-\sqrt{1-6x+x^2}}{2x}
	=1+2 x+6 x^2+22 x^3+90 x^4+394 x^5+\cdots.
$$
The $n^{th}$ little  Schr\"{o}der number $\tilde{D}(n)$ counts the number of Schr\"{o}der paths from  $[0,0]$ to $[n,n]$ without northeast 
steps on the diagonal $y=x$ whose ordinary generating function equals 
$$
\tilde{D}(x)=
\sum_{n\geq 0}\tilde{D}_nx^n= \frac{1 + x - \sqrt{1 - 6x + x^2} }{4x}
=1+ x+3 x^2+11 x^3+45 x^4+197 x^5+\cdots.
$$

Here, we adopt the same definition of paired pattern for  Delannoy paths. For example, in Figure \ref{fig:Schorder}, $L=EDNDDNNEDE\in \mathcal S_7$. $ps_L(1,2)=ENNE=P_4$ and $ps_L(2,3)=NNEE=P_6$, that is, $P_4\text{-mch}(L)=P_6\text{-mch}(L)=1$. It matches our observation: $L$ crosses the diagonal $y=x$ `vertically' once and there is one east step above $y=x+1$.
\fig{1.3}{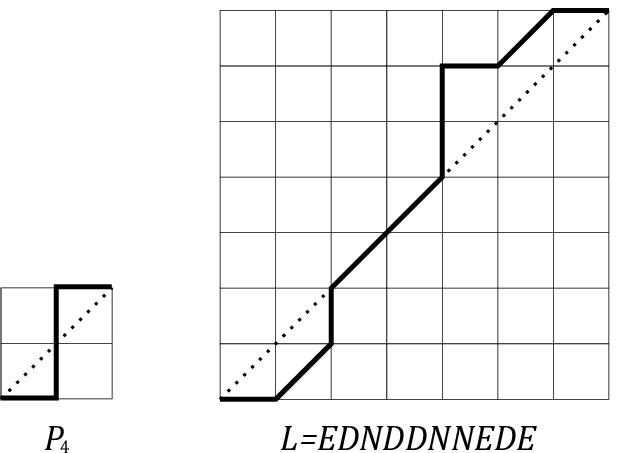}{$L=EDNDDNNEDE\in\mathcal S_7$.}

We take pattern $P_4$ as example, $P_4\text{-mch}(L)$ is the number of times $L$ crosses the diagonal $y=x$ vertically. We shall study the ordinary generating function
$$
	FS_{4}(x,t):=1+\sum_{n\geq 1}t^n\sum_{L\in\mathcal S_n}x^{P_4\text{-mch}(L)}.
$$
We split the discussion into two cases. Case 1 are the paths in $\mathcal{S}_n$ 
that  start with a north step and Case 2 are the path in $\mathcal{S}_n$ that start with an east step or a northeast step.
We define 
\begin{equation*}
GS_4(x,t):=\sum_{n\geq 1}t^n\sum_{L\in \mathcal S_n\text{ starting with $E$ or $D$}}x^{P_4\text{-mch}(L)}
\end{equation*}
and
\begin{equation*}
HS_4(x,t):=\sum_{n\geq 1}t^n\sum_{L\in \mathcal S_n\text{ starting with $N$}}x^{P_4\text{-mch}(L)}.
\end{equation*}
Clearly,
$$
FS_4(x,t)=1+GS_4(x,t)+HS_4(x,t).
$$
We obtain following formulas based on recursion on where is the first time the path starting with `$E$' or `$D$' crosses the diagonal $y=x$ from bottom to top.
$$
GS_4(x,t)=\left(D(t)-\frac{1}{1-t}\right)(xHS_4(x,t)+1)+\frac{t}{1-t}(HS_4(x,t)+1)
$$
Similarly, we consider where is the first time a path starting with a north step and having no northeast steps on the diagonal crosses the diagonal `horizontally'.
$$
HS_4(x,t)=\left(\tilde{D}(t)-1\right)(GS_4(x,t)+1)
$$
Solving for $GS_4(x,t)$, we have
$$
GS_4(x,t)=-\frac{(t-1)D(t)((\tilde{D}(t)-1)x+1)+(\tilde{D}(t)-1)x-t+1}{(\tilde{D}(t)-1)x(D(t)(t-1)+1)-2t+1} 
$$

Then we have
\begin{eqnarray*}
FS_4(x,t)&=&1+GS_4(x,t)+HS_4(x,t)\\
&=&1+GS_4(x,t)+\left(\tilde{D}(t)-1\right)\left(GS_4(x,t)+1\right)\\
&=&\tilde{D}(t)(GS_4(x,t)+1)\\
&=& \frac{2}{3+\sqrt{1-6t+t^2}-\frac{2(x-1)}{t-1}+t(x-1)-3x+\sqrt{1-6t+t^2}x}.
\end{eqnarray*}

A few initial terms of $FS_4(x,t)$ are
\begin{eqnarray*}
FS_4(x,t)&=&1+3t+(x+12)t^2+(11x+52)t^3+(x^2+84x+236)t^4\\
&&+(19x^2+556x+1108)t^5+(x^3+220x^2+3428x+5340)t^6+\cdots.
\end{eqnarray*}

By setting $x=0$ in $FS_4(x,t)$, we obtain the generating function of 
the number of  Delannoy paths that do not cross the diagonal vertically,
\begin{eqnarray*}
	FS_4(0,t)&=&\frac{(t-1)\left(-1+3t+\sqrt{1-6t+t^2}\right)}{t^2\left(3-t+\sqrt{1-6t+t^2}\right)}\\
	&=&1+3t+12t^2+52t^3+236t^4+1108t^5+5340t^6+\cdots.
\end{eqnarray*}
The sequence $1,3,12,52,236,\cdots$ does not appear in the OEIS \cite{Slo}.

Finally, one can study pattern matching for paired patterns in both lattice paths and  Delannoy paths for patterns $P$ of length $\geq 6$. For example, based on Definitions 1 and 2 and Theorem 3 and 4, one can 
obtain geometric interpretations for the number of $P$-matches in a path $L$. 

\fig{1.2}{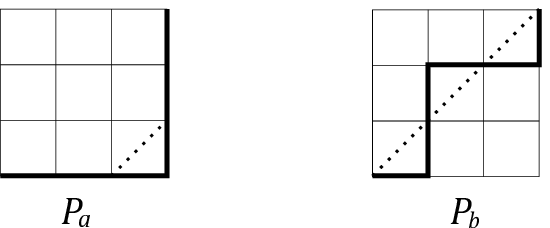}{Examples of two patterns in $\mathcal L_3$}

For example, consider the two patterns $P_a$ and $P_b$ are pictured in Figure \ref{fig:long}. Note that $P_a$ has one east step below $y=x-2$ and $P_b$ has a vertical crossing immediately followed by a horizontal crossing. For any path $L\in\mathcal L_n$, $P_a\text{-mch}(L)$ can be interpreted as the number of east steps of $L$ below $y=x-2$ and $P_b\text{-mch}(L)$ can be interpreted as the number of such pairs of crossings of $L$.

\end{document}